\documentclass[a4paper,12pt]{amsart}
\usepackage{float} 
\usepackage{soul}
\usepackage{comment}
\usepackage[T1]{fontenc}
\usepackage[utf8]{inputenc}
\usepackage[english]{babel}
\usepackage{etoolbox}
\usepackage[T1]{fontenc}
\usepackage{booktabs}
\usepackage{subfigure}
\usepackage{amsmath, amssymb, amsthm, amsfonts}
\usepackage{mathrsfs}
\usepackage{tikz}
\usepackage{graphicx}    
\usepackage{caption}     
\usepackage{subcaption}

\usetikzlibrary{arrows}
\usetikzlibrary{matrix}
\usepackage[left=1.8cm, right=1.8cm, top=2cm]{geometry}
\usepackage{enumerate}
\usepackage{hyperref}
\usepackage{algorithm}
\usepackage{algpseudocode}
\makeatletter
\newenvironment{breakablealgorithm}
  {
   \begin{center}
     \refstepcounter{algorithm}
     \hrule height.8pt depth0pt \kern2pt
     \renewcommand{\caption}[2][\relax]{
       {\raggedright\textbf{\ALG@name~\thealgorithm} ##2\par}
       \ifx\relax##1\relax
         \addcontentsline{loa}{algorithm}{\protect\numberline{\thealgorithm}##2}
       \else
         \addcontentsline{loa}{algorithm}{\protect\numberline{\thealgorithm}##1}
       \fi
       \kern2pt\hrule\kern2pt
     }
  }{
     \kern2pt\hrule\relax
   \end{center}
  }
\makeatother

\date{}
\usepackage[font=small,labelfont=bf]{caption}
\theoremstyle{plain}

\newtheorem{theorem}{Theorem}[section] 
\newtheorem{lemma}{Lemma}[section]

\newtheorem{proposition}{Proposition}[section]

\nonstopmode\numberwithin{equation}{section}
\theoremstyle{definition}
 
\newtheorem{rema}{Remark}[section]

\newtheorem*{ques*}{Question}

\makeatletter
\def\tagform@#1{\maketag@@@{\ignorespaces#1\unskip\@@italiccorr}}
\let\orgtheequation\theequation
\def\theequation{(\orgtheequation)}
\makeatother
\let\orgautoref\autoref

\renewcommand{\autoref}[1]{\def\equationautorefname{}\orgautoref{#1}}

\usepackage{fancyhdr}

\makeatletter
\def\@setauthors{%
  \begingroup
  \trivlist
  \centering\footnotesize \@topsep30\p@\relax
  \advance\@topsep by -\baselineskip
  \item\relax
  \andify\authors
  \def\\{\protect\linebreak}%
  \textbf{\authors}%
  \endtrivlist
  \endgroup
}
\makeatother

\begin{document}

\title[Randomized Krylov-Projected Iterative Regularization]{Randomized Krylov-Projected Iterated Tikhonov Regularization for Large-Scale Ill-posed Problems 
Under A Posteriori Stopping Rule} 


\author{Ravi Verma$^{\dagger}$}
\thanks{$^{\dagger}$Department of Mathematics, Indian Institute of Technology Roorkee, Roorkee, Uttarakhand, 247667, India. Email: \texttt{ravi\_v@ma.iitr.ac.in}} 

\author{Harshit Bajpai$^{\ddagger}$} 
\thanks{$^{\ddagger}$Department of Mathematics, Indian Institute of Technology Roorkee, Roorkee, Uttarakhand, 247667, India. Email: \texttt{harshit\_b@ma.iitr.ac.in}} 

\author{Ankik Kumar Giri$^{\intercal}$} 
\thanks{$^{\intercal}$Department of Mathematics, Indian Institute of Technology Roorkee, Roorkee, Uttarakhand, 247667, India. Email: \texttt{ankik.giri@ma.iitr.ac.in}}


\maketitle
\begingroup
\renewcommand{\thefootnote}{}
\footnotetext{




\textbf{Funding.}
Ravi Verma was supported by the University Grants Commission (UGC), Government of India (ID No.: 241620074745). Harshit Bajpai gratefully acknowledges the Ministry of Education, Government of India, and the Indian Institute of Technology Roorkee for funding this research through a Ph.D fellowship.

\medskip
}
\endgroup
\begin{quote}
{{\em \bf Abstract.}} We introduce two novel randomized iterative regularization frameworks, termed \texttt{RIGKT} and \texttt{RIAT}, for solving large-scale linear ill-posed inverse problems governed by systems of equations. The proposed methods combine randomized iterated Tikhonov regularization with Krylov subspace projection techniques, utilizing Golub--Kahan bidiagonalization for general rectangular systems (\texttt{RIGKT}) and Arnoldi decomposition for square systems (\texttt{RIAT}). Unlike existing deterministic schemes that rely on fixed iteration counts, our framework incorporates randomized equation selection, an adaptive step-size strategy, and a global, discrepancy-based a posteriori early-stopping rule tailored specifically to the stochastic setting. We present a comprehensive regularization analysis establishing Bregman-distance monotonicity, finite termination, exact-data convergence, and pathwise stability under noise. Furthermore, we prove that the stopped iterates converge almost surely and in the mean-square sense to the true solution, establishing a rigorous regularization property. To the best of our knowledge, this is the first theoretical framework to simultaneously account for randomization, Krylov-subspace dimension reduction, and implementable early stopping. Numerical experiments involving two-dimensional X-ray computed tomography (CT) and image deblurring  demonstrate that \texttt{RIGKT} and \texttt{RIAT} reliably reconstruct structural features across various noise regimes.

\end{quote}

\vspace{.3cm}

\noindent


\vspace{0.8mm}

\section{Introduction}
This work is concerned with the numerical solution of a class of ill-posed inverse problems modeled by the system of equations
\begin{equation}\label{model_equ:1}
\mathcal{T}_iu=v_i, \quad  i=1,\ldots,q.
\end{equation}
Here, each $\mathcal{T}_i$ denotes a bounded linear operator mapping the Hilbert space $\mathcal{U}$ into $\mathcal{V}_i$.  We assume that these operators are not continuously invertible, a common characteristic of inverse problems. Systems of this type occur in numerous practical applications including imaging sciences, geophysical exploration and machine learning,  see~\cite{bottou2018optimization,engl1996regularization,groetsch1984theory,natterer2001mathematics,scherzer2009variational}. Unless specified otherwise, these spaces are equipped with their standard inner products $\langle \cdot, \cdot \rangle$, which induce the associated norms $\|\cdot\|$.
We also consider a  proper, lower semi-continuous, and strongly convex functional $\mathfrak{f}: \mathcal{U} \to \mathbb{R}\cup\{+\infty\}$, with effective domain 
\begin{equation}\label{effective_domain}
\operatorname{dom}(\mathfrak{f}):=\{u\in\mathcal{U} : \mathfrak{f}(u) < \infty\}.
\end{equation}
We assume that \eqref{model_equ:1} admits a solution $u^\dagger$ in $\operatorname{dom}(\mathfrak{f})$, where the functional $\mathfrak{f}$ incorporates any prior knowledge regarding the expected properties of the sought solution.\\ 
In real-life applications the exact-data $\{v_1,\ldots,v_q\}$ is not available. Instead, one only has access to noisy data $\{v_1^{\delta_1},\ldots, v_{q}^{\delta_q}\}$ satisfying $\|v_{i}^{\delta_i} -v_i\| \leq \delta_i$ for each $i$, where $\delta_i$ represents the noise level of the data in the space $\mathcal{V}_i$. The total noise is given by
\[\delta : = \sqrt{\delta_1^2 + \cdots + \delta_q^2}.\]
For a true solution $u^\dagger\in\mathcal{U}$ of \eqref{model_equ:1} corresponding to the exact-data $v_i:=\mathcal{T}_iu^\dagger$, $i=1,\ldots,q$, our aim is to recover a good approximation of $u^\dagger$ using the noisy data  $v_i^{\delta_i}$ such that 
\begin{equation}\label{f_min_sol_exp}
\mathfrak{f}(u^\dagger)= \min\{\mathfrak{f}(u): \mathcal{T}_iu=v_i,\; i=1, \ldots,q\}. 
\end{equation}
Due to the inherent ill-posed nature of \eqref{model_equ:1}, reconstruction of  $u^\dagger$ from noisy observations demands regularization techniques to ensure a stable and meaningful  approximate solution. Numerous regularization strategies have been proposed for the treatment of ill-posed inverse problems, see  \cite{arridge2019solving,benning2018modern,engl1996regularization,hertrich2025learning,ito2014inverse,kaltenbacher2008iterative,scherzer2009variational} and the references therein for instance. An important class of these techniques falls under the category of Tikhonov-type variational regularization \cite{engl1996regularization,scherzer2009variational}, which are classical regularization techniques for  addressing ill-posed inverse problems. They stabilize the solution by incorporating an additional regularization term that balances data fidelity against solution regularity. As an iterative variant, a nonstationary iterated Tikhonov regularization scheme was introduced in \cite{jin2014fast}, which is given by
\begin{equation}\label{Jin&Lu2014}
\begin{cases}
u_{k}^{\delta} = \underset{u\in\mathcal{U}}{\arg\min} \{\mathfrak{f}(u)-\langle \zeta_{k}^{\delta},u\rangle\}, \\
\zeta_{k+1}^{\delta} = \zeta_{k}^{\delta}
- t_k^{\delta} \mathcal{T}^{*}(\gamma_k + \mathcal{T}\mathcal{T}^{*})^{-1}
(\mathcal{T}u_{k}^\delta-v^{\delta}).
\end{cases}
\end{equation}
Here $\zeta_0^\delta:=0$, $\{\gamma_k\}_{k\geq0}$ is a preassigned sequence of strictly positive real numbers and $t_k^\delta$ denotes the step size.  The operator $\mathcal{T}:\mathcal{U}\to \mathcal{V}_1 \times \ldots \times \mathcal{V}_q :=\mathcal{V}$ defined by 
\[\mathcal{T}u:=(\mathcal{T}_1u,\ldots, \mathcal{T}_qu),\quad u\in \mathcal{U}\]
is a bounded linear operator with adjoint  $\mathcal{T}^*$ 
and the product space $\mathcal{V}$ carries the canonical inner product induced by the inner products on its component spaces $\mathcal{V}_i$. Through an appropriate choice of the functional $\mathfrak{f}$, the method \eqref{Jin&Lu2014} can incorporate prior information and promote specific structural properties of the sought solution. 
As a special case, when $\mathfrak{f}(u)=\frac{1}{2}\|u\|^2$ and $t_k^\delta=1 \;\forall\;k\geq0$, the method \eqref{Jin&Lu2014} becomes the classical nonstationary iterated Tikhonov regularization~\cite{hanke1998nonstationary}. It is known from \cite{jin2014fast} that the method \eqref{Jin&Lu2014}
converges rapidly when the sequence $\{\gamma_k\}_{k\geq0}$ is chosen as a geometrically decreasing sequence of strictly positive real numbers. However, for \eqref{model_equ:1}, which can be a large-scale problem,  method \eqref{Jin&Lu2014} becomes highly demanding as it requires the use of all  $q$ equations simultaneously.

Recently, stochastic optimization techniques \cite{bajpai2025stochastic,bajpai2024hanke,bergou2025stochastic,huang2025early,jahn2020discrepancy,jin2026stochastic,jin2019regularizing,jin2020convergence,lu2022stochastic} have emerged as a prominent tool for handling large-scale inverse problems. By utilizing only a randomly selected equation or a subset of equations at each iteration, stochastic methods significantly reduce computational cost while maintaining good approximation quality. Motivated by these developments, a stochastic approach called \emph{Randomized iterated Tikhonov regularization} (\texttt{RITR}) is developed in \cite{jin2025randomized}, which is given by  
\begin{equation}\label{Jin&Lu2025}
\begin{cases}
u_{k}^{\delta} = \underset{u\in \mathcal{U}}{\arg\min} \{ \mathfrak{f}(u) - \langle \zeta_{k}^{\delta}, u \rangle \}, \\
\zeta_{k+1}^{\delta} = \zeta_{k}^{\delta}
- t_k^{\delta} \mathcal{T}_{I_k}^{*}
(\gamma_k + \mathcal{T}_{I_k}\mathcal{T}_{I_k}^{*})^{-1}
(\mathcal{T}_{I_k} u_{k}^{\delta} - v_{I_k}^{\delta}),
\end{cases}
\end{equation}
where for a subset $I = \{i_1, \ldots, i_b\} \subset \{1, \ldots, q\}$ containing $b$ elements, $I_k$ is a batch of $b$ equations chosen randomly at $k$-th step, $\mathcal{V}_I = \mathcal{V}_{i_1} \times \dots \times \mathcal{V}_{i_b}$ and  $v_{I} = (v_{i_1}, \dots, v_{i_b}) \in  \mathcal{V}_I$. The operator $\mathcal{T}_{I} : \mathcal{U} \to \mathcal{V}_I$ is defined as $\mathcal{T}_I u := (\mathcal{T}_{i_1} u, \dots, \mathcal{T}_{i_b} u)$ with $\mathcal{T}_I^*$ as its adjoint. Given the information on the noise level $\delta_i,\;i=1,\ldots,q$, the step sizes are chosen according to the rule
\begin{equation}\label{RITR_step_size}
t_k^\delta =
\begin{cases}
\min\left\{\frac{\mu_0\big\|\big(\gamma_k+\mathcal{T}_{I_k}\mathcal{T}_{I_k}^*\big)^{-1/2}r_k^\delta\big\|^2}{\big\|\mathcal{T}_{I_k}^*\big(\gamma_k+\mathcal{T}_{I_k}\mathcal{T}_{I_k}^*\big)^{-1}r_k^\delta\big\|^2},\;\mu_1\right\}\quad \text{if}\; \gamma_k\big\|\big(\gamma_k+\mathcal{T}_{I_k}\mathcal{T}_{I_k}^*\big)^{-1/2}r_k^\delta\big\|^2>\tau\delta_{I_k}^2,\\
0\qquad\qquad\qquad\qquad\qquad\qquad\quad\quad\text{otherwise},
\end{cases}
\end{equation}
where $r_k^\delta:= \mathcal{T}_{I_k}u_k^\delta-v_{I_k}^\delta$. It is shown that for an appropriate selection of  parameters $\tau>1,\; \mu_0>0 \; \text{ and }\mu_1>0$, the method \ref{Jin&Lu2025} together with step-size rule \ref{RITR_step_size}, converges to a solution of \eqref{model_equ:1} satisfying \eqref{f_min_sol_exp}. According to \ref{RITR_step_size}, the step-size at the $k$-th iteration is set to zero, whenever the discrepancy condition
\begin{equation}\label{DP_I_k}
\gamma_k \|(\gamma_k I + T_{I_k}T_{I_k}^{*})^{-1/2} r_k^\delta\|^2 \leq \tau\delta_{I_k}^2
\end{equation}
is satisfied for the randomly selected batch $I_k$. Thus, the step-size selection rule given by \ref{RITR_step_size} embodies the principle of an \emph{a posteriori} stopping rule. The analysis in \cite{jin2025randomized} shows that,  almost surely, the condition \ref{DP_I_k} is eventually satisfied for every admissible batch once the iteration index is sufficiently large. However, at a given iteration step $k$, the rule \ref{RITR_step_size} tests only the batch $I_k$ selected at that step. Consequently, $t_k^\delta = 0$
verifies the satisfaction of the discrepancy condition for the batch $I_k$ alone and does not verify the corresponding conditions hold simultaneously for all the batches. Therefore, although \cite{jin2025randomized} establishes an almost sure finite termination property, the randomized iterations do not contain an explicit global verification step. In particular, it does not provide an implementable global \emph{a posteriori} stopping criterion that terminates the method precisely when the discrepancy conditions hold simultaneously for all batches.\\
Moreover, in real-world applications, the operator equations in \eqref{model_equ:1} are first discretized, resulting in large-scale matrix systems. Thus, although the method \eqref{Jin&Lu2025} is  highly efficient, its implementation would still require solving linear systems involving high-dimensional matrices. Consequently,  dimension reduction techniques are necessary to make the method feasible for large-scale inverse problems.

To address these challenges, we incorporate Krylov subspace projection methods within the proposed frameworks. In particular, for rectangular matrices we employ Golub--Kahan bidiagonalization, while for square matrices we utilize the Arnoldi decomposition. By incorporating these decompositions into \eqref{Jin&Lu2025}, we propose two novel iterative regularization methods. For rectangular matrices, we approximate the discretized operators by their low-rank Golub--Kahan approximations, which gives our first  method, named the Randomized Iterated Golub--Kahan--Tikhonov (\texttt{RIGKT}) method, as
\begin{equation}\label{G_K_it_scheme_proposed}
\begin{cases}
u^{\delta}_{n,k} = \underset{u_n\in\mathcal{U}_n}{\arg\min} \{ \mathfrak{f}_n(u_n) - \langle \zeta^{\delta}_{n,k}, u_n \rangle \}, \\
\zeta^{\delta}_{n,k+1} = \zeta^{\delta}_{n,k} - t^{\delta}_k V_{l,i_k} B_{l+1,l,i_k}^{*} ( \gamma_k I + B_{l+1,l,i_k} B_{l+1,l,i_k}^{*} )^{-1} U_{l+1,i_k}^{*} ( \mathcal{T}_{i_k,n}^l u^{\delta}_{n,k} - v^{\delta}_{i_k,n}),
\end{cases}
\end{equation}
where $\zeta_{n,0}^\delta := 0$, $V_{l,i_k},\; B_{l+1,l,i_k}$ and $U_{l+1,i_k}$ are matrices involved in $l\ll n$ steps of Golub--Kahan bidiagonalization process for the discretized operator $\mathcal{T}_{i_k,n}$  chosen randomly at $k$-th step, $\mathcal{T}_{i_k,n}^l$ is the Golub--Kahan approximation for  $\mathcal{T}_{i_k,n}$ and $t_k^\delta$ denotes the step-size at iteration $k$. Here $i_k\in \{1,\ldots,q\}$ represents the index chosen randomly at the $k$-th iteration whereas $n$ denotes the discretization level of the underlying equations. This should not be confused with the dimension of the matrices. For a detailed description of the method, we refer the reader to Section~\ref{section_proposed_methods}.\\
Similarly, for square matrices, we approximate the discretized operators by their low-rank Arnoldi approximation. This gives our second method, named the Randomized Iterated Arnoldi--Tikhonov (\texttt{RIAT}) method, as
\begin{equation}\label{Arnoldi_it_scheme_proposed}
\begin{cases}
u_{n,k}^\delta = \underset{u_n \in \mathcal{U}_n}{\arg\min} \{\mathfrak{f}_n(u_n) - \langle \zeta_{n,k}^\delta, u_n \rangle \}, \\
\zeta_{n,k+1}^\delta = \zeta_{n,k}^\delta - t_k^\delta \tilde{V}_{l,i_k} H_{l+1,l,i_k}^{*} (\gamma_k I + H_{l+1,l,i_k} H_{l+1,l,i_k}^{*})^{-1} \tilde{V}_{l+1,i_k}^{*} (\tilde{\mathcal{T}}_{i_k,n}^l u_{n,k}^\delta - v_{i_k,n}^\delta),
\end{cases}
\end{equation}
where, $\zeta_{n,0}^\delta := 0$, $\tilde{V}_{l,i_k}$ and $H_{l+1,l,i_k}$ are the matrices involved in $l\ll n$ steps of Arnoldi process for the discretized operator $\mathcal{T}_{i_k,n}$ chosen randomly at $k$-th step, $\tilde{\mathcal{T}}_{i_k,n}^l$ is Arnoldi approximation for  $\tilde{\mathcal{T}}_{i_k,n}$ and $t_k^\delta$ is the step size at $k$-th step. The subscripts $i_k$ and $n$ have the same meaning as in the \texttt{RIGKT} method. For detailed description, see Section~\ref{section_proposed_methods}.\\
The schemes given in \eqref{G_K_it_scheme_proposed} and \eqref{Arnoldi_it_scheme_proposed} inherit the computational advantages of both randomization and Krylov subspace projection. Randomization reduces the number of equations used at each iteration, while the Golub--Kahan and Arnoldi decompositions replace high-dimensional matrix operations by computations involving only low-rank matrices. Hence, \texttt{RIGKT} and \texttt{RIAT} provide computationally efficient variants of \texttt{RITR} for large-scale inverse problems.\\
To mitigate the issue of a global \emph{a posteriori} criterion, we build on the strategy developed in the analysis of the stochastic mirror descent \cite{huang2025early} and the stochastic heavy ball method  \cite{gu2026posteriori}. We formulate \emph{a posteriori} criterion for methods~\eqref{G_K_it_scheme_proposed} and~\eqref{Arnoldi_it_scheme_proposed} and develop corresponding algorithms that intrinsically embed the stopping criterion into the iterative framework. These algorithms verify the discrepancy condition after a fixed number of iterations for all the component equations and terminates the method at the first index at which they are satisfied simultaneously, yielding an implementable stopping rule with low computational cost that retains the regularization property.

To the best of our knowledge, this is the first work that integrates randomized iterated Tikhonov regularization with Golub--Kahan bidiagonalization and Arnoldi decomposition while incorporating a global discrepancy-based \emph{a posteriori} stopping criterion.

\subsection{Contributions.}
Our primary contributions are the following:
\begin{itemize}
\item We propose two randomized iterative regularization methods, namely \texttt{RIGKT} and \texttt{RIAT} for solving large-scale linear ill-posed problems. These methods incorporate Golub--Kahan bidiagonalization and Arnoldi decomposition respectively.
\item Building on the global early-stopping framework of \cite{huang2025early}, we develop discrepancy-based \texttt{RIGKT} and \texttt{RIAT} algorithms that incorporate adaptive step sizes into the underlying stochastic schemes. We prove that both algorithms terminate after finitely many iterations.
\item We prove that the proposed methods possess the regularization property. In particular, we show that the iterates generated by \texttt{RIGKT} and \texttt{RIAT} converge almost surely to the solution of corresponding projected problems. We emphasize that due to randomness in the iterates as well as in the stopping index along with the Krylov projection techniques, our analysis changes drastically and requires novel ideas.
\item We validate our theoretical results through numerical simulations. In particular we demonstrate the effectiveness of our approaches on large-scale medical imaging problems, specifically focusing on reconstruction in two-dimensional X-ray Computed Tomography and the image deblurring problem.  
\end{itemize}
\subsection{Organization.}
The rest of the paper is structured as follows. Section~\ref{section_Preliminaries} introduces the preliminary material from convex and functional analysis and discusses the finite-dimensional discretization of \ref{model_equ:1}. Section~\ref{section_proposed_methods} derives the proposed \texttt{RIGKT} and \texttt{RIAT} methods by incorporating Golub--Kahan bidiagonalization and the Arnoldi process, respectively. Section~\ref{section_con_analysis_G_K} presents an algorithm for the \texttt{RIGKT} method based on a global \emph{a posteriori} stopping criterion and establishes its convergence and regularization properties. Section~\ref{section_con_analysis_Arnoldi} presents the corresponding implementation and convergence analysis for the \texttt{RIAT} method. Due to its close similarity with the analysis of \texttt{RIGKT}, only the main results are presented and the detailed proofs are omitted. Section~\ref{section_numerics} reports numerical experiments illustrating the effectiveness of the proposed methods. Finally, Section~\ref{conclusion} summarizes the main conclusions and outlines directions for future research.

\section{Preliminaries}\label{section_Preliminaries}
In Subsections~\ref{subsec:convex} and~\ref{subsec:functional}, we briefly review the fundamental concepts and results from convex analysis and functional analysis that are used throughout this paper. A detailed description of the finite-dimensional discretization of the large-scale system~\eqref{model_equ:1} is presented in Subsection~\ref{subsec:dis}. For comprehensive treatments of convex analysis and functional analysis, the reader is referred to \cite{BauschkeCombettes2011,RyuYin2022,zualinescu2002convex} and \cite{conway2019course,nair2021functional}, respectively. The discretization framework adopted in this work follows the approach presented in \cite{engl1996regularization}.

\subsection{Tools from convex analysis}\label{subsec:convex}
A convex functional $\mathfrak{f}:\mathcal{U}\to\mathbb{R}\cup\{+\infty\}$ is called proper if its effective domain defined by~\ref{effective_domain}
is non-empty. The subdifferential of $\mathfrak{f}$ at a point $u\in\operatorname{dom}(f)$ is defined as
\[\partial \mathfrak{f}(u):=\{\zeta\in \mathcal{U}: \mathfrak{f}(\bar{u})-\mathfrak{f}(u) \geq \langle\zeta,\bar{u}-u\rangle \quad \forall\;  \bar{u} \in \mathcal{U}\}.\]
If $\partial \mathfrak{f}(u)\neq\emptyset$, then $\mathfrak{f}$ is called subdifferentiable at $u$ and each element $\zeta\in \partial \mathfrak{f}(u)$ is referred to as a subgradient at $u$.  
For a given $u\in \operatorname{dom}(\mathfrak{f})$ and $\zeta\in\partial \mathfrak{f}(u)$, the Bregman distance induced by $\mathfrak{f}$ at a point $u$ in the direction of $\zeta$ is defined as 
\begin{equation}\label{bregman_formula}
D_{\mathfrak{f}}^\zeta(\bar{u},u):=\mathfrak{f}(\bar{u})-\mathfrak{f}(u)-\langle\zeta,\bar{u}-u\rangle,\quad \forall\; \bar{u}\in \operatorname{dom}(\mathfrak{f}).
\end{equation}
A proper functional $\mathfrak{f}:\mathcal{U}\to\mathbb{R}\cup\{+\infty\}$ is said to be $\nu$-strongly convex for  $\nu>0$, if 
\[\mathfrak{f}\left(\lambda u+(1-\lambda)\bar{u}\right)+\nu \lambda(1-\lambda)\|u-\bar{u}\|^2\leq \lambda \mathfrak{f}(u)+(1-\lambda)\mathfrak{f}(\bar{u})\]
for all  $u,\bar{u}\in\operatorname{dom}(\mathfrak{f})$ and $\lambda\in[0,1]$. Such functionals satisfy the following coercive condition
\begin{equation}\label{bregman_norm_relation}
D_{\mathfrak{f}}^{\zeta}(\bar{u},u)\geq\nu\|\bar{u}-u\|^2,\quad \forall\; \bar{u}, u \in \operatorname{dom}(\mathfrak{f}), \text{ and } \zeta\in \partial \mathfrak{f}(u).
\end{equation}
This estimate will be used to control the error between iterates and exact solutions in terms of Bregman distances.
The Legendre-Fenchel conjugate of $\mathfrak{f}$ is a convex functional $\mathfrak{f}^*:\mathcal{U}\to\mathbb{R}\cup\{+\infty\}$ defined by 
\[\mathfrak{f}^*(\zeta):=\sup_{u\in \mathcal{U}}\{\langle\zeta,u\rangle-\mathfrak{f}(u)\},\quad\zeta\in \mathcal{U}.\]
For every convex, proper and lower semi-continuous functional $\mathfrak{f}:\mathcal{U}\to \mathbb{R}\cup\{+\infty\}$, the following relation always holds
\begin{equation}\label{subdifferential_relation}
\zeta\in \partial \mathfrak{f}(u)\iff u\in\partial \mathfrak{f}^*(\zeta)\iff \mathfrak{f}(u)+\mathfrak{f}^*(\zeta)=\langle\zeta,u\rangle.
\end{equation}
Consequently,
\begin{equation*}\label{bregman_relation}
D_{\mathfrak{f}}^\zeta(\bar{u},u)=\mathfrak{f}(\bar{u})+\mathfrak{f}^*(\zeta)-\langle\zeta,\bar{u}\rangle.
\end{equation*}
In addition, if $\mathfrak{f}$ is $\nu$-strongly convex, then $\operatorname{dom}(\mathfrak{f}^*)=\mathcal{U}$, $\mathfrak{f}^*$ is Frechet differentiable with $ \zeta\in \partial\mathfrak{f}(u) \iff \nabla\mathfrak{f}^*(\zeta) = u $ and the Lipschitz gradient $\nabla \mathfrak{f}^*$ satisfies
\begin{equation}\label{taylor_first_order_approximation_fenchel_dual}
\begin{aligned}
\mathfrak{f}^*(\eta)-\mathfrak{f}^*(\zeta) - \langle \eta -\zeta, \nabla \mathfrak{f}^*(\zeta) \rangle &\leq \frac{1}{4\nu} \|\eta-\zeta\|^2, \quad \forall \, \zeta, \eta \in \mathcal{U},\\
\|\nabla \mathfrak{f}^*(\zeta)-\nabla \mathfrak{f}^*(\eta)\| &\leq \frac{\|\zeta-\eta\|}{2\nu}, \quad \forall \, \zeta, \eta \in \mathcal{U}.
\end{aligned}
\end{equation}

\subsection{Tools from functional analysis}\label{subsec:functional}
Let $A:\mathcal X\to\mathcal Y$ be a bounded linear operator between Hilbert spaces $\mathcal X$ and $\mathcal Y$.
Then, for any $\alpha>0$, the operator  $(\alpha I+AA^*)$ is self-adjoint, strictly positive and invertible. Consequently, by the spectral theorem 
there exists a unique bounded, self-adjoint and strictly positive operator $(\alpha I+AA^*)^{1/2}$, called the positive square root of $(\alpha I+AA^*)$, such that
\[(\alpha I+AA^*)^{1/2}(\alpha I+AA^*)^{1/2}=\alpha I+AA^*.\]
The inverse square root of $(\alpha I+AA^*)$ is defined as $(\alpha I+AA^*)^{-1/2}:=\bigl((\alpha I+AA^*)^{1/2}\bigr)^{-1}$, which is again a bounded self-adjoint operator. Moreover, if $0<\alpha_1\leq\alpha_2$, then in the sense of positive operators, we have $(\alpha_1I+AA^*)\leq(\alpha_2I+AA^*)$. Consequently, $(\alpha_2I+AA^*)^{-1/2}\leq(\alpha_1I+AA^*)^{-1/2}$. Furthermore, the following operator norm estimates hold.
\begin{equation}\label{operator_bounds}
\big\|(\alpha I+AA^*)^{-1/2}\big\|\leq
\frac{1}{\sqrt{\alpha}},\quad
\big\|A^*(\alpha I+AA^*)^{-1/2}\big\|\leq1.
\end{equation}
\\
We will also use the following elementary property of matrices. Let $Q\in \mathbb{R}^{m\times n}$ be a matrix with orthonormal columns, equivalently $Q^*Q=I_{n}$, where $I_n$ is an $n\times n$ identity matrix. Then, for every $x\in\mathbb{R}^{n}$ and $y\in \mathbb{R}^m$, we have $\|Qx\|\leq \|x\|$ and $\|Q^*y\|\leq\|y\|$, respectively. Consequently, $\|Q\|\leq 1$
and $\|Q^*\|\leq 1$.

\subsection{Discretization of the problem}\label{subsec:dis}
In this subsection, we consider the discretization  of \eqref{model_equ:1} together with the functional $\mathfrak{f}$.\\
Let $\mathcal{U}_1 \subset \mathcal{U}_2 \subset \cdots \subset \mathcal{U}_n \subset \cdots \subset \mathcal{U}$ be a sequence of finite-dimensional subspaces whose union is dense in $\mathcal{U}$ with $\dim(\mathcal{U}_n) < \infty$. Define the projection operators $P_n:\mathcal{U}\to \mathcal{U}_n$ and $Q_{i,n}:\mathcal{V}_i \to \mathcal{V}_{i,n}:=\mathcal{T}_i(\mathcal{U}_n)$, and the inclusion operator $\mathfrak{I}_n:\mathcal{U}_n \hookrightarrow \mathcal{U}.$ These operators together with \eqref{model_equ:1} and corresponding noisy counterpart $\mathcal{T}_iu^\delta=v_i^\delta$, $i=1,\ldots,q$, give
\[Q_{i,n}\mathcal{T}_i \mathfrak{I}_n P_n u = Q_{i,n} v_i \quad\text{and}\quad Q_{i,n} \mathcal{T}_i \mathfrak{I}_n P_n u^\delta=Q_{i,n} v_i^\delta.\]
Define the operators $\mathcal{T}_{i,n}:\mathcal{U}_n \to \mathcal{Y}_{i,n}$, given by $ \mathcal{T}_{i,n} : = Q_{i,n} \mathcal{T}_i \mathfrak{I}_n$ and finite-dimensional vectors 
\[v_{i,n} := Q_{i,n} v_i,\quad v^{\delta}_{i,n} := Q_{i,n} v^{\delta}_i,\quad u_n := P_n u,\quad \text{and } u^{\delta}_n := P_n u^{\delta}.\]
Now, we can identify $\mathcal{T}_{i,n}$ with a matrix in $\mathbb{R}^{n_1\times n_2}$  with $\operatorname{dim}(\mathcal{U}_n)=n_2$ and $\operatorname{dim}(\mathcal{V}_{i,n})=n_1$, $v_{i,n}$ and  $v_{i,n}^\delta $ with elements in $\mathbb{R}^{n_1}$ and $ u_n$,   $u_n^\delta$ with elements in $\mathbb{R}^{n_2}.$
Thus, we obtain the  following linear system of equations
\begin{equation}\label{disc_model_exactdata}
\mathcal{T}_{i,n} u_n=v_{i,n}\quad  i=1,\ldots,q.
\end{equation}
\begin{equation}\label{disc_model_inexactdata}
\mathcal{T}_{i,n} u_n^\delta=v_{i,n}^\delta, \quad i=1,\ldots,q.
\end{equation}
From now on, for each $i\in\{1,\ldots,q\}$, we consider $\mathcal{T}_{i,n}$ to be a rectangular matrix that represents a discretization of $\mathcal{T}_i$.\\
Furthermore, we denote the restriction of the functional $\mathfrak{f}$ to the subspace $\mathcal{U}_n$ by $\mathfrak{f}_n$, \[\mathfrak{f}_n : \mathcal{U}_n \to \mathbb{R} \cup \{+\infty\}.\]
Thus, the discretized version of \eqref{f_min_sol_exp} can be written as 
\begin{equation}\label{disc_f_min_sol_exp}
\mathfrak{f}_n(u_n^\dagger)=\min\{\mathfrak{f}_n(u_n):\mathcal{T}_{i,n}u_n=v_{i,n}, \; i=1,\ldots,q\},
\end{equation}
where $u_n^\dagger$ is an $n_2$-dimensional approximation  of $u^\dagger.$\\
It should be noted that the discretization process introduces discretization error. The propagation of this error may be estimated using the framework of \cite{natterer1977regularization}. However, since this work mainly concerns the construction and analysis of the proposed \texttt{RIGKT} and \texttt{RIAT} methods, we do not pursue a detailed discretization-error analysis here.\\ 
We conclude this subsection by choosing a suitable basis $\{e_1,e_2,\ldots,e_{n_2}\}$ for the discrete space $\mathcal{U}_n$. Then, any element $u_n \in \mathcal{U}_n$ admits a unique representation
\begin{equation*}
u_n = \sum_{j=1}^{n_2} u_j^{(n)} e_j.
\end{equation*}
Consequently, $u_n$ can be identified with its coefficient vector $\mathbf{u}_n = [u_1^{(n)}, \dots, u_{n_2}^{(n)}]^* \in \mathbb{R}^{n_2}$.
This identification allows the finite-dimensional operator equations in \ref{disc_model_exactdata} and \ref{disc_model_inexactdata} to be treated as matrix equations and therefore provides the algebraic formulation required for their numerical implementation. Henceforth, we do not distinguish notationally between an element $u_n$ of $\mathcal{U}_n$ and its corresponding coefficient vector.\\
In the case where $\{e_j\}_{j=1}^{n_2}$ constitutes an orthonormal basis, the norms may be scaled in a way such that $\|\mathbf{x}_n\|_2 = \|x_n\|$, where $\|\cdot\|_2$ denotes the usual Euclidean norm. Nevertheless, since certain discretization schemes utilize non-orthonormal bases, this equality may not hold in general. Hence, following the approach in \cite{ramlau2019error}, we assume the existence of positive constants $\kappa$ and $\tilde{\kappa}$, independent of $n$, such that the following norm equivalence holds
\begin{equation*}
\kappa \|\mathbf{u}_n\|_2 \leq \|u_n\| \leq \tilde{\kappa} \|\mathbf{u}_n\|_2.
\end{equation*}

This inequality is satisfied in many practical applications and prevents the coordinate representation from becoming increasingly ill-conditioned as the discretization is refined, see \cite{de1978practical,goodman2016discrete}.

Now, based on the discretized systems \eqref{disc_model_exactdata} and \eqref{disc_model_inexactdata}, an approximation of $u_n^\dagger$ can be obtained via the discretized version of the iterative scheme \eqref{Jin&Lu2025}, given by
\begin{equation}\label{sol_via_disc}
\begin{cases}
u_{n,k}^\delta = \underset{u_n \in \mathcal{U}_n}{\arg\min} \left\{ \mathfrak{f}_n(u_n) - \langle \zeta_{n,k}^\delta, u_n \rangle \right\}, \\
\zeta_{n,k+1}^\delta = \zeta_{n,k}^\delta - t_k^\delta \mathcal{T}_{i_k,n}^* (\gamma_k I + \mathcal{T}_{i_k,n} \mathcal{T}_{i_k,n}^*)^{-1} (\mathcal{T}_{i_k,n} u_{n,k}^\delta - v_{i_k,n}^\delta),
\end{cases}
\end{equation}
initialized with $\zeta_{n,0}^\delta := 0$. In this formulation, $u_{n,k}^\delta$ and $\zeta_{n,k}^\delta$ represent the $k$-th iterates restricted to the subspace $\mathcal{U}_n$. The operator $\mathcal{T}_{i_k,n}$ is the discretized forward operator randomly selected at each iteration $k$, while $t_k^\delta$ and $\{\gamma_k\}_{k\geq 0}$ denote the step-sizes and  a monotonically decreasing sequence of positive real numbers, respectively.

\section{The proposed methods}\label{section_proposed_methods}
In this section, we build upon the discretized scheme \ref{sol_via_disc} to derive the proposed methods given in \eqref{G_K_it_scheme_proposed} and \eqref{Arnoldi_it_scheme_proposed}. Construction of these schemes depends on the dimensions of discretized operators $\mathcal{T}_{i,n}$, $i=1,\ldots,q$. When $\mathcal{T}_{i,n} \in \mathbb{R}^{n_1\times n_2}$ is a rectangular matrix with $n_1\neq n_2$, we employ Golub--Kahan bidiagonalization and in the square case ($n_1 = n_2$), we use a computationally more efficient Arnoldi decomposition. We begin with the case of rectangular matrices.

\subsection{\texttt{RIGKT} method} For high-dimensional rectangular matrices,  Golub--Kahan bidiagonalization is a well-established technique for reducing a large problem to a  manageable lower-dimensional form. In this subsection, we incorporate the Golub--Kahan bidiagonalization process into the discretized iterative scheme \eqref{sol_via_disc} to derive the \texttt{RIGKT} method. Detailed descriptions of Golub--Kahan bidiagonalization can be found in \cite{bjorck2024numerical,golub2013matrix}, and its use in the context of Tikhonov regularization has been explored in \cite{bianchi2026iterated,bjorck1988bidiagonalization,calvetti2003tikhonov,furchi2027iterated,gazzola2015krylov}.\\
To formulate the method, let us fix $ 1 \leq l \ll \min \{n_1,n_2\} $. Then, for each $ i \in \{1,\ldots, q\} $, applying $l$ steps of the Golub--Kahan bidiagonalization process to the matrix $\mathcal{T}_{i,n}$ and the vector $v_{i,n}^\delta$,  we get the Golub--Kahan decomposition
\begin{equation}\label{Golub_Kahan_decomposition}
\mathcal{T}_{i,n} V_{l,i} = U_{l+1,i} B_{l+1,l,i} 
\quad \text{and} \quad 
\mathcal{T}_{i,n}^* U_{l,i} = V_{l,i} B_{l,l,i}^{*}.
\end{equation}
Here, $V_{l,i} = [\dot v_{1,i}, \ldots, \dot v_{l,i}] \in \mathbb{R}^{n_2 \times l}$ and $U_{l+1,i} = [u_{1,i}, \ldots, u_{l+1,i}] \in \mathbb{R}^{n_1 \times (l+1)}$ are matrices with orthonormal columns. The matrices $U_{l,i}\in \mathbb{R}^{n_1\times l}$ and $B_{l,l,i} \in \mathbb{R}^{l \times l}$ are formed by the first $l$ columns of $U_{l+1,i} \in \mathbb{R}^{n_1\times(l+1)}$ and first $l$ rows of the lower bidiagonal matrix 
\[
B_{l+1,l,i} =
\begin{bmatrix}
\alpha_{1,i} & 0 & \cdots & 0\\
\beta_{2,i} & \alpha_{2,i} & \ddots & \vdots\\
0 & \beta_{3,i} & \ddots & 0\\
\vdots & \ddots & \ddots & \alpha_{l,i}\\
0 & \cdots & 0 & \beta_{l+1,i}
\end{bmatrix}
\in \mathbb{R}^{(l+1)\times l}.
\]
Moreover,
\begin{align}
\nonumber \operatorname{span}\{u_{1,i}, \ldots, u_{l+1,i}\} &= \operatorname{span}\{v_{i,n}^\delta,\; (\mathcal{T}_{i,n} \mathcal{T}_{i,n}^*) v_{i,n}^\delta, \ldots, (\mathcal{T}_{i,n} \mathcal{T}_{i,n}^*)^l v_{i,n}^\delta\},\\ \nonumber
\operatorname{span}\{\dot v_{1,i}, \ldots, \dot v_{l,i}\} &=
\operatorname{span}\{\mathcal{T}_{i,n}^* v_{i,n}^\delta,\; (\mathcal{T}_{i,n}^* \mathcal{T}_{i,n}) \mathcal{T}_{i,n}^* v_{i,n}^\delta, \ldots, (\mathcal{T}_{i,n}^* \mathcal{T}_{i,n})^{l-1} \mathcal{T}_{i,n}^* v_{i,n}^\delta\}.
\end{align}
We assume that no breakdown occurs during the first $l$ Golub--Kahan steps, so that all nontrivial entries of $B_{l+1,l,i}$ are nonvanishing, i.e., $B_{l+1,l,i}$ is of rank-$l$. Based on \eqref{Golub_Kahan_decomposition}, we define the rank-$l$ Golub--Kahan approximation of $\mathcal{T}_{i,n}$ as
\begin{equation}\label{Golub_Kahan_approx}
\mathcal{T}_{i,n}^l := U_{l+1,i} B_{l+1,l,i} V_{l,i}^{*} \in \mathbb{R}^{n_1 \times n_2}.
\end{equation}
It is evident that for $l = \min\{n_1, n_2\}$, the approximation becomes exact, i.e., $\mathcal{T}_{i,n}^l = \mathcal{T}_{i,n}$.  Therefore, for sufficiently large $l$, we can assume that $\mathcal{T}_{i,n}^l$ is a good approximation of $\mathcal{T}_{i,n}$. Now, for the fixed $l$, we write the Golub--Kahan projected systems
\begin{equation}\label{dis_G_K_mod_eq_exact_da}
\mathcal{T}_{i,n}^l u_n = v_{i,n}, \quad i=1,\ldots,q,
\end{equation}
\begin{equation}\label{dis_G_K_mod_eq_inexact_da}
\mathcal{T}_{i,n}^lu_n^\delta = v_{i,n}^\delta, \quad i=1,\ldots,q,
\end{equation}
and assume that, for the fixed discretization level $n$,
\begin{equation}\label{GK_projected_solution_set}
\mathcal{S}_{n,l}:=\{u_n\in\operatorname{dom}(\mathfrak f_n):\mathcal T_{i,n}^{l}u_n=v_{i,n},\ i=1,\ldots,q\}
\end{equation}
is non-empty. Consequently, we denote by $u_{n,l}^{\dagger}$ the $\mathfrak f_n$-minimizing solution of \eqref{dis_G_K_mod_eq_exact_da}, i.e.,
\begin{equation}\label{GK_projected_minimizer}
u_{n,l}^{\dagger}:=\arg\min\{\mathfrak f_n(u_n):u_n\in\mathcal S_{n,l}\}.
\end{equation}
Our objective is to replace $\mathcal{T}_{i,n}$ in \eqref{sol_via_disc} with its Golub--Kahan approximation $\mathcal{T}_{i,n}^l$. To achieve this, we first exploit the structural properties of $\mathcal{T}_{i,n}^l$.
Utilizing the definitions of $\mathcal{T}_{i,n}^l$, $U_{l+1,i}$ and $V_{l,i}$, one can observe that
\begin{equation*}
\mathcal{T}_{i,n}^{l*} (\gamma_k I + \mathcal{T}_{i,n}^l \mathcal{T}_{i,n}^{l*})^{-1} = V_{l,i} B_{l+1,l,i}^{*} (\gamma_k I + B_{l+1,l,i} B_{l+1,l,i}^{*})^{-1} U_{l+1,i}^{*}.
\end{equation*}
Since $1\leq l\ll \min\{n_1,n_2\}$, the inversion of $(\gamma_k I + B_{l+1,l,i} B_{l+1,l,i}^{*})$ takes place in the low-dimensional space $\mathbb{R}^{(l+1)\times (l+1)}$. Thus, this identity simplifies the evaluation of the operator on the left-hand side. Consequently, substituting $\mathcal{T}_{i,n}^l$ for $\mathcal{T}_{i,n}$ in \eqref{sol_via_disc} and applying this relation yields the proposed iterative scheme \eqref{G_K_it_scheme_proposed}.\\
We now turn to the case of square matrices to derive the iterative scheme  
\eqref{Arnoldi_it_scheme_proposed}.

\subsection{\texttt{RIAT} method}
As a special case, when $n_1 = n_2 = n$ such that $\dim(\mathcal{U}_n) = \dim(\mathcal{V}_{i,n}) = n$, the discretized operators in \eqref{disc_model_exactdata}  can be represented as square matrices. In this setting, the Arnoldi process provides a natural one-sided Krylov subspace reduction. We incorporate this reduction into the discretized iterative scheme \eqref{sol_via_disc} to derive the \texttt{RIAT} method. Detailed description of the Arnoldi process can be found in \cite{ramlau2019error,saad2003iterative}, while its use in the context of Tikhonov regularization has been explored in  \cite{bianchi2025convergence,gazzola2014generalized}.\\
Let us fix $1 \leq l \ll n$. Then, for each $i \in \{1, \ldots, q\}$, applying $l$ steps of the Arnoldi process to the matrix $\mathcal{T}_{i,n}$ and the vector $v_{i,n}^\delta$, we get the Arnoldi decomposition
\begin{equation}\label{Arnoldi_dec}
\mathcal{T}_{i,n} \tilde{V}_{l,i} = \tilde{V}_{l+1,i} H_{l+1,l,i}.
\end{equation}
Here, the matrix $\tilde{V}_{l+1,i} = [\tilde{v}_{1,i}, \ldots, \tilde v_{l+1,i}] \in \mathbb{R}^{n \times (l+1)}$ has orthonormal columns and constitutes an orthonormal basis for the Krylov subspace
\begin{equation*}
\mathcal{K}_{l+1}(\mathcal{T}_{i,n}, v_{i,n}^\delta) := \operatorname{span}\{v_{i,n}^\delta, \mathcal{T}_{i,n} v_{i,n}^\delta, \ldots, (\mathcal{T}_{i,n})^l v_{i,n}^\delta\}.
\end{equation*}
The matrix $H_{l+1,l,i} \in \mathbb{R}^{(l+1)\times l}$ is an upper Hessenberg matrix, where all entries below the first subdiagonal vanishes. Based on \eqref{Arnoldi_dec}, we define the rank-$l$ Arnoldi approximation of $\mathcal{T}_{i,n}$ by
\begin{equation*}\label{Arnoldi_Appr}
\tilde{\mathcal{T}}_{i,n}^l := \tilde{V}_{l+1,i} H_{l+1,l,i} \tilde{V}_{l,i}^{*} \in \mathbb{R}^{n \times n}.
\end{equation*}
Notably, since $\tilde{\mathcal{T}}_{i,n}^l = \mathcal{T}_{i,n} \tilde{V}_{l,i} \tilde{V}_{l,i}^*$, the approximation becomes exact when $l = n$. Following the Golub--Kahan case, for the fixed $l$, we write the Arnoldi projected systems
\begin{equation}\label{dis_Arnoldi_exact_data}
\tilde{\mathcal{T}}_{i,n}^l u_n = v_{i,n}, \quad i=1,\ldots,q,
\end{equation}
\begin{equation*}
\tilde{\mathcal{T}}_{i,n}^l u_n^\delta = v_{i,n}^\delta,\quad i=1,\ldots,q,
\end{equation*}
and assume that, for fixed discretization level $n$,
\begin{equation*}\label{Arnoldi_projected_solution_set}
\tilde{\mathcal S}_{n,l}:=\{u_n\in\operatorname{dom}(\mathfrak f_n):\tilde{\mathcal T}_{i,n}^{l}u_n=v_{i,n},\ i=1,\ldots,q\}
\end{equation*}
is non-empty. Consequently, we denote by $\tilde u_{n,l}^{\dagger}$ the $\mathfrak f_n$-minimizing solution of \eqref{dis_Arnoldi_exact_data}, i.e.,
\begin{equation*}\label{Arnoldi_projected_minimizer}
\tilde u_{n,l}^{\dagger}:=\arg\min\{\mathfrak f_n(u_n):u_n\in\tilde{\mathcal S}_{n,l}\}.
\end{equation*}
Utilizing the definitions of $\mathcal{T}_{i,n}^l$ and $\tilde{V}_{l+1,i}$, we observe that
\begin{equation*}
\tilde{\mathcal{T}}_{i,n}^{l*} (\gamma_k I + \tilde{\mathcal{T}}_{i,n}^l \tilde{\mathcal{T}}_{i,n}^{l*})^{-1} = \tilde{V}_{l,i} H_{l+1,l,i}^{*} (\gamma_k I + H_{l+1,l,i} H_{l+1,l,i}^{*})^{-1} \tilde{V}_{l+1,i}^{*}.
\end{equation*}
Since $1\leq l\ll n$, the inversion of $(\gamma_k I + H_{l+1,l,i} H_{l+1,l,i}^{*})$ takes place in a low-dimensional space $\mathbb{R}^{(l+1)\times(l+1)}$. Therefore, this  identity simplifies the evaluation of the operator on the left-hand side. Consequently, substituting $\tilde{\mathcal{T}}_{i,n}^l$ for $\mathcal{T}_{i,n}$ in \eqref{sol_via_disc} and applying this relation yields the proposed iterative scheme \eqref{Arnoldi_it_scheme_proposed}.
\begin{rema}
For each component equation, the \texttt{RIGKT} method requires the construction of two Krylov subspaces. In contrast, the \texttt{RIAT} method employs only a single Arnoldi Krylov subspace, represented by $\tilde V_{l+1,i}$. Consequently, for square systems, \texttt{RIAT} requires fewer basis vectors to be generated, orthogonalized, and stored. This one-sided Krylov construction may therefore significantly reduce both the computational cost and memory requirements, particularly when the Krylov dimension or the number of component equations is large.
\end{rema}

In the next section, we discuss the  implementation and convergence analysis of the \texttt{RIGKT} method. 


\section{Convergence Analysis}\label{section_con_analysis_G_K}
This section studies the convergence behavior of the proposed \texttt{RIGKT} method~\eqref{G_K_it_scheme_proposed}. First, we develop a residual-dependent step-size  based on an \emph{a posteriori} discrepancy principle. We then demonstrate the monotonicity of the iterates generated by the algorithm, finite termination of the method, and regularization attributes. To begin, we define the step-size rule for~\eqref{G_K_it_scheme_proposed} as follows. 

Let us denote 
\[A_{k,i}:=\gamma_kI+B_{l+1,l,i}B_{l+1,l,i}^*, \qquad h_{k,i}^\delta := (A_{k,i})^{-1/2}U_{l+1,i}^*(\mathcal{T}_{i,n}^lu_{n,k}^\delta-v_{{i},n}^\delta),\] 
and $\delta_{i,n}$ denotes the noise level for the discretized data, i.e., $\|v_{i,n}^\delta-v_{i,n}\| \leq \delta_{i,n}$.
Moreover, 
$i_k \in \{1, \ldots, q\}$ represents a random index drawn at iteration $k$.
Now, using these notation we define
\begin{equation}\label{step_size_Golub_Kahan}
t_k^\delta := 
\begin{cases} 
\min\left\{ \frac{\mu_0 \big\|h_{k,i_k}^\delta\big\|^2}{\big\|V_{l,i_k} B_{l+1,l,i_k}^{*}\big(A_{k,i_k}\big)^{-1/2}h_{k,i_k}^\delta\big\|^2},\; \mu_1 \right\} & \text{if}\;\; \gamma_k \big\|h_{k,i_k}^\delta\big\|^2 > \tau^2 \delta_{i_k,n}^2, \\
0 & \text{otherwise},
\end{cases}
\end{equation}
where the parameters $\mu_0>0, \; \mu_1>0,$ and $\tau>1$ are chosen such that 
\begin{equation}\label{parameter_condition}
\mathfrak{C}_0 := 1-\frac{1}{\tau}-\frac{\mu_0}{4\nu} >0,
\end{equation}
and the sequence $\{\gamma_j\}_{j\ge0}$ is chosen such that
\begin{equation}\label{sequence_condition}
\sum_{j=0}^{\infty}\frac{1}{\gamma_j}=\infty.
 \end{equation}
For notational simplicity, throughout this section we denote the total discretized noise by $\delta$. We also omit the discretization index $n$ from the noise level, writing $\delta_i$ instead of $\delta_{i,n}$, with the understanding that $\|v_{i,n}^{\delta} - v_{i,n}\| \le \delta_i$.\\  
We observe that, whenever $ \gamma_k \|h_{k,i_k}^\delta\|^2 > \tau^2 \delta_{i_k}^2$, the active step-size is bounded away from zero. Indeed, using the bounds \eqref{operator_bounds} and $\|V_{l,i_k}\|\leq1$, we obtain
\[
\|V_{l,i_k} B_{l+1,l,i_k}^{*}(A_{k,i_k})^{-1/2}h_{k,i_k}^\delta\|\; \leq\; \| B_{l+1,l,i_k}^{*}(A_{k,i_k})^{-1/2}h_{k,i_k}^\delta\|\; \leq \;\|h_{k,i_k}^\delta\|.
\]
Consequently,
\[\frac{\mu_0 \big\|h_{k,i_k}^\delta\big\|^2}{\big\|V_{l,i_k} B_{l+1,l,i_k}^{*}\big(A_{k,i_k}\big)^{-1/2}h_{k,i_k}^\delta\big\|^2}\geq\mu_0\] and hence every active update satisfies $t_k^\delta\geq \min\{\mu_0,\mu_1\}>0$. We use the above step-size selection rule in the following algorithm for implementing \texttt{RIGKT}. It is adapted from the framework developed in \cite{huang2025early} for the stochastic mirror-descent.
\begin{breakablealgorithm}
\caption{ \texttt{RIGKT} with \emph{a posteriori} stopping}
\label{algo:G_K_noisy_data}
\begin{algorithmic}[1]
\State Input: $\tau > 1$, $\mu_0 > 0$, $\mu_1 > 0$ and $\zeta_{n,0} \in \mathcal{U}_n$.
\State Compute $u_{n,0} = \arg\min_{u_n \in \mathcal{U}_n}\bigl\{\mathfrak{f}_n(u_n) - \langle \zeta_{n,0}, u_n\rangle\bigr\}$ and set
\Statex \hspace{\algorithmicindent} 
$\zeta_{n,0,0}^\delta = \zeta_{n,0}$,\quad
$u_{n,0,0}^\delta = u_{n,0}$
\For{$k = 0, 1,  \ldots,$} (\textbf{Outer loop})
\State Pick a fixed positive integer $m_k$.
\For{$m = 0, \ldots, m_k - 1$} (\textbf{Inner loop})
\State Draw $i_{k,m} \in \{1,\ldots,q\}$ uniformly at random and compute
\begin{align*}
h_{k,m}^\delta &= (A_{i_{k,m}})^{-1/2}
U_{l+1,i_{k,m}}^{*}(\mathcal{T}_{i_{k,m},n}^lu_{n,k,m}^\delta-v_{i_{k,m},n}^\delta),
\\
g_{k,m}^\delta &= 
V_{l,i_{k,m}} B_{l+1,l,i_{k,m}}^{*}
(A_{i_{k,m}})^{-1/2}h_{k,m}^\delta,
\end{align*}
\Statex \hspace{\algorithmicindent}
\hspace{0.5cm} where $A_{i_{k,m}}:=(\gamma_{k,m}I+ B_{l+1,l,i_{k,m}}B_{l+1,l,i_{k,m}}^*).$ 
\State Set step-size
\[
t_{k,m}^\delta =
\begin{cases}
\min\!\left\{
\dfrac{\mu_0\,\|h_{k,m}^\delta\|^2}
{\|g_{k,m}^\delta\|^2},\;
\mu_1
\right\}
& \text{if } \gamma_{k,m}\|h_{k,m}^\delta\|^2 
 > \tau^2\delta_{i_{k,m}}^2, \\[8pt]
0 & \text{otherwise.}
\end{cases}
\]
\State Set
$\zeta_{n,k,m+1}^\delta = \zeta_{n,k,m}^\delta 
- t_{k,m}^\delta g_{k,m}^\delta$,
and compute
\[
u_{n,k,m+1}^\delta = \arg\min_{u_n \in \mathcal{U}_n}
\bigl\{\mathfrak{f}_n(u_n) - \langle \zeta_{n,k,m+1}^\delta, 
u_n\rangle\bigr\}.
\]
\EndFor
\State Outer update: Set
$\zeta_{n,k+1}^\delta = \zeta_{n,k,m_k}^\delta$  \text{ and }
$u_{n,k+1}^\delta = u_{n,k,m_k}^\delta$.
\State Stopping check: \textbf{If}
\begin{equation}\label{DP_G_K}
\gamma_{k+1}
\bigl\|h_{k+1,i}^\delta\bigr\|^2
\leq \tau^2\delta_{i}^2
\quad \forall\, i \in \{1,\ldots,q\},
\end{equation}
\Statex \hspace{\algorithmicindent} then output $u_{n,k_\delta}^\delta:=u_{n,k+1}^\delta$ and stop
\Statex \hspace{\algorithmicindent} \textbf{else} compute
\begin{align*}
\Delta_{k+1}^\delta &=
\Bigl\{ i \in \{1,\ldots,q\} :
\gamma_{k+1}
\bigl\|h_{k+1,i}^\delta\bigr\|^2
> \tau^2\delta_{i}^2 \Bigr\},\\
g_{k+1}^\delta 
&= \sum_{i \in \Delta_{k+1}^\delta}
V_{l,i} B_{l+1,l,i}^{*}
\bigl(A_{k+1,i}\bigr)^{-1/2}h_{k+1,i}^\delta,\\
\Phi_{k+1}^\delta 
&= \sum_{i \in \Delta_{k+1}^\delta}
\bigl\|h_{k+1,i}^\delta\bigr\|^2.
\end{align*}
\State Outer step-size: Set
\[
t_{k+1}^\delta =
\begin{cases}
\min\!\left\{
\dfrac{\mu_0\,\Phi_{k+1}^\delta}
{\|g_{k+1}^\delta\|^2}, \; \mu_1
\right\}
& \text{if } \Delta_{k+1}^\delta \neq\emptyset,\\
0
& \text{if } \Delta_{k+1}^\delta = \emptyset.\\[6pt]
\end{cases}
\]
\State Outer update: Set
$\zeta_{n,k+1,0}^\delta = \zeta_{n,k+1}^\delta 
- t_{k+1}^\delta g_{k+1}^\delta$,
and compute
\[
u_{n,k+1,0}^\delta = \arg\min_{u_n \in \mathcal{U}_n}
\bigl\{\mathfrak{f}_n(u_n) - \langle \zeta_{n,k+1,0}^\delta, 
u_n\rangle\bigr\}.
\]
\EndFor
\end{algorithmic}
\end{breakablealgorithm}
The above algorithm consists of two phases, an inner loop and an outer loop. In the inner loop, it performs a prescribed $m_k$ number of updates, each using step-size according to rule \eqref{step_size_Golub_Kahan} with $k$ replaced by $(k,m)$ and $i_k$ replaced by $i_{k,m}$. Subsequently, the outer loop verifies the satisfaction of the discrepancy condition \eqref{DP_G_K}. If this condition is unmet, the algorithm identifies the indices for which it is violated and utilizes this information for the next update. This iterative process continues until \eqref{DP_G_K} is satisfied. We define the stopping index for the above algorithm as the  smallest outer-iteration index $k_\delta$ at which the stopping criterion \eqref{DP_G_K} is satisfied.
It should be noted that the quantities $\gamma_{k,m}$ and $\gamma_k$ appearing in Algorithm~\ref{algo:G_K_noisy_data} are taken from the same prescribed sequence $\{\gamma_j\}_{j\geq0}$. The double index $(k,m)$ denotes the $m$-th
inner randomized update within the $k$-th outer cycle. Thus, $\gamma_{k,m}$
denotes the element of the sequence $\{\gamma_j\}_{j\geq0}$ used at this inner
update. Similarly, $\gamma_k$ denotes the element used at the corresponding
outer update. More precisely, if $\mathfrak{z}_{k,m}$ and $\mathfrak{z}_k$ denote the total number of updates performed up to the inner step $(k,m)$ and the outer step $k$, respectively, then $\gamma_{k,m}:=\gamma_{\mathfrak{z}_{k,m}} \text{ and } \gamma_k:=\gamma_{\mathfrak{z}_k}.$\\
In the following subsection, we establish a monotonicity result for Algorithm~\ref{algo:G_K_noisy_data} and use it to demonstrate that the algorithm terminates after a finite number of steps.

\subsection{Monotonicity and finite termination property for \texttt{RIGKT}}
\begin{lemma}\label{lemma_mon_G_k}
Consider the iterative procedure in Algorithm~\ref{algo:G_K_noisy_data}. Let the parameters $\tau>1, \mu_0>0$ be chosen so that \eqref{parameter_condition} is satisfied. Let $\hat{u}_n \in \mathcal{S}_{n,l}$, then for all $k\leq k_\delta$ and $m=0,\ldots,m_k-1$, the following monotonicity estimates hold.
\begin{equation}\label{mono_G_K_algo_1}
\Gamma_{n,k,m+1}^\delta-\Gamma_{n,k,m}^\delta\leq-\mathfrak{C}_0t_{k,m}^\delta\|h_{k,m}^\delta\|^2\quad\text{and}\quad
\Gamma_{n,k+1,0}^\delta-\Gamma_{n,k+1}^\delta\leq-\mathfrak{C}_0t_{k+1}^\delta\Phi_{k+1}^\delta,
\end{equation}
where
$\Gamma_{n,k,m}^\delta := D_{\mathfrak{f}_n}^{\zeta_{n,k,m}^\delta}(\hat{u}_n,u_{n,k,m}^\delta)$ and $\Gamma_{n,k+1}^\delta:=D_{\mathfrak{f}_n}^{\zeta_{n,k+1}^\delta}(\hat{u}_n , u_{n,k+1}^\delta).$
\end{lemma}
\begin{proof} To prove the desired result, we first establish that the Bregman distance remains bounded by its initial value, i.e.,
\begin{equation}\label{mono_analogue_G_k}
\Gamma_{n,k,m}^\delta \leq \Gamma_{n,0}, \quad \forall \, k < k_\delta,\; m = 0, \ldots, m_k,
\end{equation}
where $\Gamma_{n,0} := D_{\mathfrak{f}_n}^{\zeta_{n,0}}(\hat{u}_n, u_{n,0})$. We prove this by induction on the iteration indices which requires verification of the following two conditions:
\begin{enumerate}[(i)]
\item If $\Gamma_{n,k,m}^\delta \leq \Gamma_{n,0}$ for some $0 \leq m < m_k$, then $\Gamma_{n,k,m+1}^\delta \leq \Gamma_{n,0}$.
\item If $\Gamma_{n,k}^\delta \leq \Gamma_{n,0}$ for some $k <  k_\delta$, then $\Gamma_{n,k,0}^\delta \leq \Gamma_{n,0}$.
\end{enumerate}
We start with the proof of condition (ii).
Utilizing the definition \eqref{bregman_formula} and the relation \eqref{subdifferential_relation}, we have
\begin{equation*}
\Gamma_{n,k,0}^\delta - \Gamma_{n,k}^\delta = \mathfrak{f}_n^*(\zeta_{n,k,0}^\delta) - \mathfrak{f}_n^*(\zeta_{n,k}^\delta) - \langle \zeta_{n,k,0}^\delta - \zeta_{n,k}^\delta, \hat{u}_n \rangle.
\end{equation*}
Since $\mathfrak{f}_n$ is  $\nu$-strongly convex, by the properties of Fenchel conjugate and \eqref{taylor_first_order_approximation_fenchel_dual}, it follows that 
\begin{align}\label{lemma_1_eqn_1} 
\nonumber
\Gamma_{n,k,0}^\delta - \Gamma_{n,k}^\delta &= \mathfrak{f}_n^*(\zeta_{n,k,0}^\delta) - \mathfrak{f}_n^*(\zeta_{n,k}^\delta) - \langle \zeta_{n,k,0}^\delta - \zeta_{n,k}^\delta, \nabla \mathfrak{f}_n^*(\zeta_{n,k}^\delta) \rangle - \langle \zeta_{n,k,0}^\delta - \zeta_{n,k}^\delta, -u_{n,k}^\delta + \hat{u}_n \rangle \\
&\leq \frac{1}{4\nu} \|\zeta_{n,k,0}^\delta - \zeta_{n,k}^\delta\|^2 + \langle \zeta_{n,k,0}^\delta - \zeta_{n,k}^\delta, u_{n,k}^\delta - \hat{u}_n \rangle.
\end{align}
Now, substituting the update rule  $\zeta_{n,k,0}^\delta = \zeta_{n,k}^\delta  -t_k^\delta g_k^\delta$ and using definitions of adjoint and $\mathcal{T}_{i,n}^l$, along with the relation $\mathcal{T}_{i,n}^l \hat{u}_n = v_{i,n}$, the inner product term of \eqref{lemma_1_eqn_1} expands as
\begin{align}\label{lemma_1_eqn_2}
\nonumber
\langle -t_k^\delta g_k^\delta, u_{n,k}^\delta - \hat{u}_n \rangle 
&= -t_k^\delta \bigg\langle \sum_{i \in \Delta_k^\delta} V_{l,i} B_{l+1,l,i}^{*}
(A_{k,i})^{-1/2}h_{k,i}^\delta, u_{n,k}^\delta-\hat{u}_n \bigg\rangle \\ \nonumber
&= -t_k^\delta \sum_{i\in \Delta_k^\delta}\big\langle h_{k,i}^\delta, (A_{k,i})^{-1/2}B_{l+1,l,i}V_{l,i}^*(u_{n,k}^\delta-\hat{u}_n)\big\rangle  \\ \nonumber
&= -t_k^\delta \sum_{i\in \Delta_k^\delta}\big\langle h_{k,i}^\delta, (A_{k,i})^{-1/2}U_{l+1,i}^*(\mathcal{T}_{i,n}^lu_{n,k}^\delta-v_{i,n})\big\rangle \\
&= -t_k^\delta \sum_{i \in \Delta_k^\delta} \|h_{k,i}^\delta\|^2 -t_k^\delta \sum_{i \in \Delta_k^\delta}\big\langle h_{k,i}^\delta, (A_{k,i})^{-1/2}U_{l+1,i}^* (v_{i,n}^\delta - v_{i,n}) \big\rangle .
\end{align}
Applying  Cauchy-Schwarz inequality together with the bounds $\| v_{i,n}^\delta - v_{i,n} \| \leq \delta_{i}$, $\|U_{l+1,i}^*\|\leq1$ and \ref{operator_bounds},
yields
\begin{equation}\label{lemma_1_eqn_3}
-\langle h_{k,i}^\delta, (A_{k,i})^{-1/2}U_{l+1,i}^* (v_{i,n}^\delta - v_{i,n}) \rangle \leq \frac{\delta_i}{\sqrt{\gamma_k}} \| h_{k,i}^\delta\|.
\end{equation}
Moreover, by definition of $\Delta_k^\delta$, we have $\frac{\delta_{i}}{\sqrt{\gamma_k}} < \frac{1}{\tau} \| h_{k,i}^\delta\|$. Using this inequality along with \eqref{lemma_1_eqn_3} and \eqref{lemma_1_eqn_2} together with the definition of $\Phi_k^\delta$ and the step-size  $t_k^\delta$ in \eqref{lemma_1_eqn_1}, we obtain
\begin{align*}
\Gamma_{n,k,0}^\delta - \Gamma_{n,k}^\delta &\leq \frac{(t_k^\delta)^2}{4\nu} \|g_k^\delta\|^2 - t_k^\delta \Phi_k^\delta + \frac{t_k^\delta}{\tau} \Phi_k^\delta \\
&\leq \frac{t_k^\delta \mu_0 \Phi_k^\delta}{4\nu} - t_k^\delta \Phi_k^\delta + \frac{t_k^\delta}{\tau} \Phi_k^\delta \\
&= - \left( 1 - \frac{1}{\tau} - \frac{\mu_0}{4\nu} \right) t_k^\delta \Phi_k^\delta\\
&= -\mathfrak{C}_0 t_k^\delta \Phi_k^\delta.
\end{align*}
This implies $\Gamma_{n,k,0}^\delta < \Gamma_{n,k}^\delta$. Using  the assumption $\Gamma_{n,k}^\delta < \Gamma_{n,0}$, it follows that $\Gamma_{n,k,0}^\delta \leq \Gamma_{n,k}^\delta \leq \Gamma_{n,0}$, which  proves the condition~(ii).\\
Repeating the same arguments as in the proof of~(ii), we establish condition~(i). Similar to \eqref{lemma_1_eqn_1}, we have 
\[\Gamma_{n,k,m+1}^\delta-\Gamma_{n,k,m}^\delta \leq \frac{1}{4\nu}\|\zeta_{n,k,m+1}^\delta-\zeta_{n,k,m}^\delta\|^2 + \langle \zeta_{n,k,m+1}^\delta - \zeta_{n,k,m}^\delta, u_{n,k,m}^\delta -\hat{u}_n \rangle.\]
The assertion in~(i) is immediate when \(t_{k,m}^\delta=0\). So assuming the case \(t_{k,m}^\delta>0\), we observe that
\[
\big\langle
\zeta_{n,k,m+1}^{\delta}-\zeta_{n,k,m}^{\delta},
u_{n,k,m}^{\delta}-\hat u_n
\big\rangle
\leq
-\left(1-\frac{1}{\tau}\right)
t_{k,m}^{\delta}\|h_{k,m}^{\delta}\|^2 .
\]
Consequently, using the definition of $t_{k,m}^{\delta}$, we obtain
\begin{align*}
\Gamma_{n,k,m+1}^{\delta}-\Gamma_{n,k,m}^{\delta}
&\leq
\frac{(t_{k,m}^{\delta})^2}{4\nu}\|g_{k,m}^{\delta}\|^2
-\left(1-\frac{1}{\tau}\right)
t_{k,m}^{\delta}\|h_{k,m}^{\delta}\|^2  \\
&\leq
-\left(1-\frac{1}{\tau}-\frac{\mu_0}{4\nu}\right)
t_{k,m}^{\delta}\|h_{k,m}^{\delta}\|^2\\
&=-\mathfrak C_0 t_{k,m}^{\delta}\|h_{k,m}^{\delta}\|^2 .
\end{align*}
Together with the inductive assumption $\Gamma_{n,k,m}^\delta \leq \Gamma_{n,0}$, this implies $\Gamma_{n,k,m+1}^\delta \leq \Gamma_{n,k,m}^\delta \leq \Gamma_{n,0}$, thereby establishing condition (i). Finally, given that \eqref{mono_analogue_G_k} has been established, the assertion in \eqref{mono_G_K_algo_1} follows directly from the steps used in proving conditions (i) and (ii).
\end{proof}
Based on the preceding lemma, we now establish finite-time termination property of Algorithm~\ref{algo:G_K_noisy_data}.

\begin{lemma}\label{lemma_finite_termination_G_K}
Let the parameters $\tau >1  \text{ and } \mu_0>0$ be chosen such that \eqref{parameter_condition} is satisfied. Then, Algorithm~\ref{algo:G_K_noisy_data} terminates in finitely many steps. 
\end{lemma}
\begin{proof}
It is sufficient to prove that, for every fixed sample path $\{i_{k,m} : k \geq 0,\; 0 \leq m < m_k \}$, there exists a finite non-negative  integer $k_\delta$ for which the following holds.
\[\gamma_{k_\delta}
\bigl\|h_{k_\delta,i}^\delta\bigr\|^2
\leq \tau^2\delta_{i}^2
\quad \forall\, i \in \{1,\ldots,q\}.\]
We prove this by contradiction. Let us suppose that no such $k_\delta$ exists. Then, the set $\Delta_{k}^\delta$ is non-empty for every $k \geq 1$. Consequently, there must exist an index $i_k \in \Delta_k^\delta$, such that
\(\gamma_k\|h_{k,i_k}^\delta\|^2 > \tau^2 \delta_{i_k}^2.\)
Therefore, using the definition of $\Phi_k^\delta$, for each integer $k \geq 1$, we have
\begin{equation}\label{G_k_termination_1}
\gamma_{k} \Phi_k^\delta \geq \gamma_{k} \left\| h_{k,i_k}^\delta \right\|^2 > \tau^2 \delta_{i_k}^2 \geq \tau^2 \delta_{\min}^2,
\end{equation}
where $\delta_{\min} := \min \{\delta_{1}, \ldots, \delta_{q}\}$.  From the monotonicity estimates~\eqref{mono_G_K_algo_1}, for every $k \geq 1$ we observe that \[\Gamma_{n,k+1}^\delta - \Gamma_{n,k}^\delta = (\Gamma_{n,k,0}^\delta - \Gamma_{n,k}^\delta) + \sum_{m=0}^{m_k-1} (\Gamma_{n,k,m+1}^\delta - \Gamma_{n,k,m}^\delta) \leq -\mathfrak{C}_0 t_k^\delta \Phi_k^\delta - \mathfrak{C}_0t_{k,m}^\delta\|h_{k,m}^\delta\|^2 \leq -\mathfrak{C}_0 t_{k}^\delta \Phi_k^\delta   . \]
Therefore, for any fixed integer $N\geq1$,
\[ \mathfrak{C}_0 \sum_{k=1}^{N} t_{k}^\delta \Phi_{k}^\delta \leq \sum_{k=1}^{N} (\Gamma_{n,k}^\delta -\Gamma_{n,k+1}^\delta) =\Gamma_{n,1}^\delta -\Gamma_{n,N+1}^\delta \leq \Gamma_{n,1}^\delta \leq\Gamma_{n,0}^\delta=\Gamma_{n,0}<\infty.\]
Consequently, taking the limit  $N \to \infty$, we obtain
\begin{equation}\label{G_k_termination_2}
\sum_{k=1}^{\infty} t_{k}^\delta \Phi_{k}^\delta < \frac{\Gamma_{n,0}}{\mathfrak{C}_0}<\infty.
\end{equation}
Now, we show that whenever $\Delta_k^\delta \neq \emptyset$, the step-size $t_k^\delta$ is bounded away from zero. By the definition of $t_k^\delta$, we have $t_k^\delta = \min \{ \mu_0 \Phi_k^\delta / \|g_k^\delta\|^2, \mu_1 \}$. Applying Cauchy--Schwarz inequality and the operator norm bounds $\|V_{l,i}\| \leq 1$ and \eqref{operator_bounds}, we can bound $\|g_k^\delta\|^2$ as 
\begin{align*}
\|g_k^\delta\|^2 &= \bigg\| \sum_{i \in \Delta_{k}^\delta} V_{l,i}B_{l+1,l,i}^{*}(A_{k,i})^{-1/2} h_{k,i}^\delta \bigg\|^2 \leq \bigg(\sum_{i \in \Delta_{k}^\delta} \| h_{k,i}^\delta \|\bigg)^{2} \leq q \sum_{i\in \Delta_k^\delta}\| h_{k,i}^\delta\|^2 =q\Phi_k^\delta.
\end{align*}
This shows that there exists a positive constant $\bar{\mu} := \min \left\{\frac{\mu_0}{q}, \mu_1\right\} > 0$ such that $t_k^\delta \geq \bar{\mu}$. Combining this with \eqref{G_k_termination_1}, we get 
\begin{equation*}
\sum_{k=1}^N t_{k}^\delta \Phi_{k}^\delta \geq \sum_{k=1}^{N}\frac{t_k^\delta\tau^2\delta_{\min}^2}{\gamma_k}\geq \bar{\mu} \tau^2 \delta_{\min}^2 \sum_{k=1}^N \frac{1}{\gamma_k}.
\end{equation*}
Under assumption~\ref{sequence_condition}, the right-hand side of above inequality diverges as $N \to \infty$ contradicting the finite summability established in \eqref{G_k_termination_2}. Thus, we conclude that Algorithm~\ref{algo:G_K_noisy_data} terminates in finitely many steps.
\end{proof}

The above lemma establishes that Algorithm~\ref{algo:G_K_noisy_data} is well defined and always terminates after a finite number of steps, yielding a finite stopping index $k_{\delta}$. Since this index depends on the realized sequence of random indices, it is itself a random variable. We are now ready to examine the regularization properties of iterates generated by Algorithm~\ref{algo:G_K_noisy_data}.

\subsection{Regularization property of \texttt{RIGKT}}
In this subsection, our aim is to show that,  as the noise level $\delta\to0$, $u_{n,k_\delta}^\delta$ converges almost surely to the projected minimizer $u_{n,l}^{\dagger}$ defined in \eqref{GK_projected_minimizer}. For this, we will follow the following strategy: 
\begin{itemize}
\item Formulate the exact-data counterpart for Algorithm~\ref{algo:G_K_noisy_data} and establish its almost sure convergence (\emph{Convergence in exact-data case}).
\item Connect Algorithm~\ref{algo:G_K_noisy_data} to its exact-data counterpart and establish a stability result which holds along each sample path (\emph{Pathwise stability}).
\item Utilize these to establish regularization property for Algorithm~\ref{algo:G_K_noisy_data}  (\emph{Regularization property}).
\end{itemize}
We begin by defining the exact-data counterpart of Algorithm~\ref{algo:G_K_noisy_data}.

\begin{breakablealgorithm}
\caption{\texttt{RIGKT} with exact-data}
\label{algo:G_K_exact_data_exact}
\begin{algorithmic}[1]
\State Input: $\mu_0 > 0$, $\mu_1 > 0$ and $\zeta_{n,0} \in \mathcal{U}_n$.
\State Compute $u_{n,0} = \arg\min_{u_n \in \mathcal{U}_n}
\bigl\{\mathfrak{f}_n(u_n) - \langle \zeta_{n,0}, u_n\rangle\bigr\}$ and set
\Statex \hspace{\algorithmicindent}
$\zeta_{n,0,0} = \zeta_{n,0}$,\quad
$u_{n,0,0} = u_{n,0}$.
\For{$k = 0, 1, \ldots,$} (\textbf{Outer loop})
\State Pick a fixed positive integer $m_k$.
\For{$m = 0,\ldots, m_k - 1$} (\textbf{Inner loop})
\State Draw $i_{k,m} \in \{1,\ldots,q\}$ uniformly at random and compute
\begin{align*}
h_{k,m}
&= (A_{i_{k,m}})^{-1/2}
U_{l+1,i_{k,m}}^{*}(\mathcal{T}_{i_{k,m},n}^{l}u_{n,k,m}
- v_{i_{k,m},n}),
\\
g_{k,m}
&=
V_{l,i_{k,m}} B_{l+1,l,i_{k,m}}^{*}
(A_{i_k,m})^{-1/2} h_{k,m}.
\end{align*}

\State Set step-size
\[
t_{k,m} =
\begin{cases}
\min\!\left\{
\dfrac{\mu_0\,\|h_{k,m}\|^2}
{\|g_{k,m}\|^2},\;
\mu_1
\right\}
& \text{if } h_{k,m} \neq 0, \\[8pt]
0 & \text{otherwise.}
\end{cases}
\]
\State Set
$\zeta_{n,k,m+1} = \zeta_{n,k,m} - t_{k,m}g_{k,m}$,
and compute
\[
u_{n,k,m+1}
= \arg\min_{u_n \in \mathcal{U}_n}
\bigl\{\mathfrak{f}_n(u_n)
- \langle \zeta_{n,k,m+1},u_n\rangle\bigr\}.
\]
\EndFor
\State Outer update: Set $\zeta_{n,k+1} = \zeta_{n,k,m_k} \text{ and }
u_{n,k+1} = u_{n,k,m_k}.$
\State Stopping check: \textbf{If}
\[
\bigl\|h_{k+1,i}\bigr\| = 0
\quad \forall\, i \in \{1,\ldots,q\},
\]
\Statex \hspace{\algorithmicindent} then output $u_{n,k+1}$ and stop
\Statex \hspace{\algorithmicindent} \textbf{else} compute
\begin{align*}
\Delta_{k+1}
&=
\Bigl\{ i \in \{1,\ldots,q\} :
\bigl\|h_{k+1,i}\bigr\|^2 > 0 \Bigr\},
\\
g_{k+1}
&=
\sum_{i \in \Delta_{k+1}}
V_{l,i} B_{l+1,l,i}^{*}
\bigl(A_{k+1,i})^{-1/2} h_{k+1,i},
\\
\Phi_{k+1}
&=
\sum_{i \in \Delta_{k+1}}
\bigl\|h_{k+1,i}\bigr\|^2,
\end{align*}
\Statex \hspace{\algorithmicindent} where $h_{k,i}
:=
\bigl(A_{k,i})^{-1/2}
U_{l+1,i}^{*}
\bigl(\mathcal{T}_{i,n}^{l}u_{n,k}
- v_{i,n}\bigr).$
\State Outer step-size: Set
\[
t_{k+1} =
\begin{cases}
\min\!\left\{
\dfrac{\mu_0\,\Phi_{k+1}}
{\|g_{k+1}\|^2},\;
\mu_1
\right\}
& \text{if } \Delta_{k+1} \neq \emptyset, \\[8pt]
0
& \text{if } \Delta_{k+1} = \emptyset.
\end{cases}
\]
\State Outer update: Set $\zeta_{n,k+1,0}
= \zeta_{n,k+1} - t_{k+1}g_{k+1}$,
and compute
\[u_{n,k+1,0}
=
\arg\min_{u_n \in \mathcal{U}_n}
\bigl\{\mathfrak{f}_n(u_n)
- \langle \zeta_{n,k+1,0},u_n\rangle\bigr\}.\]
\EndFor
\end{algorithmic}
\end{breakablealgorithm}

We will establish that the iterates $\{u_{n,k}\}$ generated by the above algorithm, converge almost surely to a $\operatorname{dom}({\mathfrak{f}_n})$ valued random variable $u_n^*$ satisfying \eqref{dis_G_K_mod_eq_exact_da}. We refer to $u_n^*$ a random solution of \eqref{dis_G_K_mod_eq_exact_da}. To make this precise, we first introduce the 
underlying probability space on which the iterates are defined. Let $\Pi := \{1, \ldots, q\}$ denote the index set and 
$\mathcal{A}$ be the $\sigma$-algebra containing all subsets of 
$\Pi$. At each inner iteration of 
Algorithm~\ref{algo:G_K_exact_data_exact}, an index is drawn uniformly from $\Pi$. Thus, after $\sum_{j=0}^{k-1}m_j+m$ random selections, the variables $u_{n,k,m+1}$ and $\zeta_{n,k,m+1}$ depend only
on the corresponding history of sampled indices. Therefore they are random variables on the probability space $(\Pi_{k,m}, \mathcal{A}_{k,m}, \mathbb{P}_{k,m})$, where
\[\Pi_{k,m} := \underbrace{\Pi \times \cdots \times \Pi}_{\sum_{j=0}^{k-1} m_j \;+\; m\;\;\text{copies}} ,\;\; \mathcal{A}_{k,m} 
:= \underbrace{\mathcal{A} \otimes \cdots \otimes \mathcal{A}}_{\sum_{j=0}^{k-1} m_j \;+\; m\;\;\text{copies}}\]
and $\mathbb{P}_{k,m}$ is the uniform probability measure. Note that the  updates $u_{n,k+1,0}$ and $\zeta_{n,k+1,0}$ are obtained deterministically from the iterates at the end of the $k$-th inner cycle. Hence they introduce no additional randomness beyond what is already present at step 
$(k, m_k - 1)$ and are therefore measurable with respect to $\mathcal{A}_{k,m_k-1}$. The finite-dimensional measures constructed in this way are consistent, consequently the Kolmogorov extension theorem \cite{bhattacharya2007basic} ensures the existence of a unique probability 
measure $\mathbb{P}$ on the measurable space
$(\Omega, \mathcal{F}) 
:= (\Pi^{\infty},\, \mathcal{A}^{\otimes \infty})$, whose finite-dimensional marginals coincide with the measure $\mathbb{P}_{k,m}$. Hence, all random variables arising in Algorithm~\ref{algo:G_K_exact_data_exact} can be
realized on the common probability space $(\Omega,\mathcal{F},\mathbb{P})$ and the notion of almost sure convergence is understood with respect to this space.\\
The following result explicitly constructs a measurable event of probability one, on which the convergence analysis of 
Algorithm~\ref{algo:G_K_exact_data_exact} is carried out. 

\begin{lemma}\label{lemma_prob1_G_K}
Consider Algorithm~\ref{algo:G_K_exact_data_exact} and suppose the parameter $\mu_0>0$ is chosen so that 
\begin{equation}\label{parameter_condition_2}
\mathfrak{C}_1 := 1 - \frac{\mu_0}{4\nu} > 0.
\end{equation}
Let $u_n^*\in\mathcal S_{n,l}$. Then, for all integers $k \geq 0$ and $m = 0, 1, \ldots, m_k-1$, the following monotonicity property holds
\begin{equation}\label{mono_G_K_algo_2}
\Gamma_{n,k+1} \leq \Gamma_{n,k,m+1} \leq \Gamma_{n,k,m} \leq \Gamma_{n,k},
\end{equation}
where $\Gamma_{n,k,m} := D_{\mathfrak{f}_n}^{\zeta_{n,k,m}}(u_n^*, u_{n,k,m}) \text{ and } \Gamma_{n,k} := D_{\mathfrak{f}_n}^{\zeta_{n,k}}(u_n^*, u_{n,k}).$
Furthermore, the event 
\begin{equation}\label{Prob_one_event}
\begin{aligned}
\Omega_0 := \Bigg\{ \sum_{k=0}^{\infty} \Bigg( & \sum_{m=0}^{m_k-1} \sum_{i=1}^{q}\|h_{k,m,i}\|^2  + \sum_{i=1}^{q} \left\|h_{k+1,i} \right\|^2 \Bigg) < \infty \Bigg\}\\
\end{aligned}
\end{equation}
has probability one, i.e., $\mathbb{P}(\Omega_0) = 1$, where 
\begin{equation*}
h_{k,m,i}:=(\gamma_{k,m}I+B_{l+1,l,i}B_{l+1,l,i}^*)^{-1/2} U_{l+1,i}^{*} (\mathcal{T}_{i,n}^l u_{n,k,m} - v_{i,n}).
\end{equation*}
\end{lemma}
\begin{proof}
Following the proof of monotonicity estimates in Lemma~\ref{lemma_mon_G_k} for the exact-data case, for every $k\geq0$ and $m=0,\ldots,m_k-1$, we have
\[\Gamma_{n,k,m+1}-\Gamma_{n,k,m}\leq-\mathfrak{C}_1 t_{k,m} \|h_{k,m}\|^2 \; \text{ and }\; \Gamma_{n,k+1,0} - \Gamma_{n,k+1}\leq-\mathfrak{C}_1 t_{k+1}\Phi_{k+1}.\]
Moreover, using analogous arguments as in Lemma~\ref{lemma_finite_termination_G_K}, for $\bar\mu=\min \left\{\frac{\mu_0}{q}, \mu_1\right\} > 0$, we also have 
\[-t_{k,m}\|h_{k,m}\|^2\leq-\bar\mu\|h_{k,m}\|^2 \;\text{ and }\; -t_{k+1}\Phi_{k+1}\leq -\bar\mu\Phi_{k+1}.\]
Therefore, for every $k\geq0$ and $m=0,\ldots,m_k-1,$ we obtain
\begin{equation}\label{mono_G_K_algo_3}
\Gamma_{n,k,m+1}-\Gamma_{n,k,m}\leq-\mathfrak{C}_1\bar{\mu}  \|h_{k,m}\|^2 \quad \text{and}
\end{equation}
\begin{equation}\label{mono_G_K_algo_3,1}
\Gamma_{n,k+1,0}-\Gamma_{n,k+1}\leq - \mathfrak{C}_1\bar{\mu}\Phi_{k+1}.
\end{equation}
Consequently, the inequalities 
\[\Gamma_{n,k,m+1}\le \Gamma_{n,k,m} \quad \text{and} \quad \Gamma_{n,k,0}\le \Gamma_{n,k}\]
together with the observation $\Gamma_{n,k+1}=\Gamma_{n,k,m_k}$ establishes \eqref{mono_G_K_algo_2}.
It remains to prove that $\mathbb{P}(\Omega_0)=1$. For this consider the filtration $\{\mathcal{F}_{k,m}:k\geq0 \;\text{ and }\; 0\leq m<m_k\}$ defined by 
\[\mathcal{F}_{k,m}:=\sigma\big(\{i_{r,s} : 0\leq r<k \;\;,\; 0\leq s<m_r\} \cup\{i_{k,s} : 0\leq s<m\}\big).\]
That is, $\mathcal{F}_{k,m}$ denotes the $\sigma$-algebra generated by all random indices chosen before the inner step $(k,m)$.
Since the sequence $\{\Gamma_{n,k,m}\}$ is adapted to the filtration $\{\mathcal{F}_{k,m}\}$ and $i_{k,m}$ is uniformly distributed in $\{1,\ldots,q\}$, taking expectation and using the tower property of conditional expectation in \eqref{mono_G_K_algo_3}, we have 
\begin{equation*}\label{Prob_1_event_1}
\begin{aligned}
\mathbb{E}[\Gamma_{n,k,m+1}]-\mathbb{E}[\Gamma_{n,k,m}]&
\leq -\mathfrak{C}_1\bar\mu\,
    \mathbb E[\|h_{k,m}\|^2]\\
&=- \mathfrak{C}_1 \bar{\mu} \mathbb{E}\left[\mathbb{E}\left[\|h_{k,m}\|^2| \mathcal{F}_{k,m}\right]\right]\\
&= -\mathfrak{C}_1 \bar{\mu}\mathbb{E}\bigg[\sum_{i=1}^{q}\|(\gamma_{k,m}+B_{l+1,l,i}B_{l+1,l,i}^*)^{-1/2}U_{l+1,i}^{*}(\mathcal{T}_{i,n}^l u_{n,k,m} - v_{i,n})\|^2 \times \frac{1}{q}\bigg] \\
&= -\frac{\mathfrak{C}_1\bar{\mu}}{q}\mathbb{E}\bigg[\sum_{i=1}^{q}\|h_{k,m,i}\|^2\bigg]\\
&= -\frac{\mathfrak{C}_1\bar{\mu}}{q}\mathbb{E}[S_{k,m}],
\end{aligned}
\end{equation*}
where $S_{k,m}:=\sum_{i=1}^{q}\|h_{k,m,i}\|^2$.
For fixed $k$,   taking sum over $m=0,\ldots, m_k-1$ in the above inequality, yields 
\begin{equation}\label{Prob_1_event_2}
\sum_{m=0}^{m_k-1}\big(\mathbb{E}[\Gamma_{n,k,m+1}]-\mathbb{E}[\Gamma_{n,k,m}]\big)=\mathbb{E}[\Gamma_{n,k+1}] - \mathbb{E}[\Gamma_{n,k,0}] \leq -\frac{\mathfrak{C}_1\bar{\mu}}{q}\sum_{m=0}^{m_k-1}\mathbb{E}[S_{k,m}].
\end{equation}
Moreover, by \eqref{mono_G_K_algo_3,1} we also have 
\begin{equation}\label{h_1}
    \mathbb{E}[\Gamma_{n,k+1,0}]-\mathbb{E}[\Gamma_{n,k+1}]\leq - \mathfrak{C}_1\bar{\mu}\mathbb{E}[\Phi_{k+1}] \leq \frac{-\mathfrak{C}_1\bar{\mu}}{q}\mathbb{E}[\Phi_{k+1}].
\end{equation}
Adding \eqref{h_1} to \eqref{Prob_1_event_2}, we get
\[\mathbb{E}[\Gamma_{n,k+1,0}] - \mathbb{E}[\Gamma_{n,k,0}] \leq - \frac{\mathfrak{C}_1\bar{\mu}}{q} \left( \sum_{m=0}^{m_k-1} \mathbb{E}[S_{k,m}] + \mathbb{E}[\Phi_{k+1}] \right).\]
Thus, for every fixed positive integer $N$,
\begin{align*}
\frac{\mathfrak{C}_1\bar{\mu}}{q} \sum_{k=0}^{N} \left( \sum_{m=0}^{m_k-1} \mathbb{E}[S_{k,m}] + \mathbb{E}[\Phi_{k+1}] \right) &\leq \sum_{k=0}^{N}\big(\mathbb{E}[\Gamma_{n,k,0}]-\mathbb{E}[\Gamma_{n,k+1,0}]\big) \\
& \leq \mathbb{E}[\Gamma_{n,0,0}] - \mathbb{E}[\Gamma_{n,N+1,0}]   \leq \mathbb{E}[\Gamma_{n,0,0}] = \Gamma_{n,0,0}<\infty.
\end{align*}
Letting $N\to\infty$, by  monotone convergence theorem it follows that 
\[\frac{\mathfrak{C}_1\bar{\mu}}{q}\mathbb{E}\left[\sum_{k=0}^{\infty}\left\{\sum_{m=0}^{m_k-1}S_{k,m} + \Phi_{k+1} \right\} \right] <\infty \text{ and hence } \mathbb{E}\left[\sum_{k=0}^{\infty}\left\{\sum_{m=0}^{m_k-1} S_{k,m} + \Phi_{k+1}\right\}\right] <\infty.\]
Consequently,
\[\sum_{k=0}^{\infty}\left\{\sum_{m=0}^{m_k-1} S_{k,m} + \Phi_{k+1}\right\}<\infty \text{ almost surely.}\]
By virtue of Algorithm~\ref{algo:G_K_exact_data_exact}, we note that $\Phi_{k+1} = \sum_{i=1}^{q}\|h_{k+1,i}\|^2$. Therefore we conclude that $\mathbb{P}(\Omega_0)=1$. This completes the proof.
\end{proof}

We next establish the almost sure convergence of Algorithm~\ref{algo:G_K_exact_data_exact}. The proof relies on the  following general convergence result. For details, see \cite{jin2014fast,jin2013landweber}.

\begin{proposition}\label{proposition_G_K}
Let $\mathfrak{f}_n:\mathcal{U}_n\to\mathbb{R}\cup\{+\infty\}$ be a proper, lower semi-continuous and strongly convex functional. Moreover, assume that $\{u_{n,k}\}\subset \operatorname{dom}(\mathfrak{f}_n)$ and $\{\zeta_{n,k}\}\subset \mathcal{U}_n$ be such that the following holds.
\begin{enumerate}[(i)]
\item For all $k\geq0$, $\zeta_{n,k}\in\partial \mathfrak{\mathfrak{f}}_n(u_{n,k})$.
\item For any $\hat{u}_n\in \mathcal S_{n,l}$, the sequence $\big\{D_{\mathfrak{f}_n}^{\zeta_{n,k}}(\hat{u}_n, u_{n,k})\big\}$ is convergent.
\item For each  $i=1,\ldots,q$,   $\lim_{k\to\infty}\|\mathcal{T}_{i,n}^l u_{n,k}- v_{i,n}\|=0$. 
\item  There exists a subsequence $\{k_w\}$ of positive integers with $k_w\to \infty$ as $w\to\infty$, such that for any $\hat{u}_n \in \mathcal S_{n,l}$, 
\begin{equation}
\lim_{z\to\infty}\sup_{w\geq z}|\langle\zeta_{n,k_w}-\zeta_{n,k_z},u_{n,k_w}-\hat{u}_n\rangle|=0.
\end{equation}
\end{enumerate}
Then, there exists $u_n^* \in \mathcal S_{n,l}$ such that 
\[\lim_{k\to \infty} D_{\mathfrak{f}_n}^{\zeta_{n,k}}(u_n^*,u_{n,k})=0.\]
Moreover, if in addition $\zeta_{n,k+1}-\zeta_{n,k} \in \operatorname{Ran}(\mathcal{T}_{1,n}^{l*}) \oplus \cdots \oplus \operatorname{Ran}(\mathcal{T}_{q,n}^{l*})$ for all k, then $u_n^* = u_{n,l}^{\dagger}$, where $u_{n,l}^{\dagger}$ is the unique $\mathfrak{f}_n$-minimizing solution of \eqref{dis_G_K_mod_eq_exact_da}. 
\end{proposition}
\begin{theorem}\label{Theorem_G_K_exact}
Consider Algorithm~\ref{algo:G_K_exact_data_exact} and assume the parameter $\mu_0>0$ is chosen so that \eqref{parameter_condition_2} holds. Let there exists a fixed positive integer $M$ with $m_k\leq M$ for all k. Then, the iterates generated by Algorithm~\ref{algo:G_K_exact_data_exact} converges almost surely to the $\mathfrak f_n$-minimizing solution $u_{n,l}^\dagger$ of \eqref{dis_G_K_mod_eq_exact_da}, i.e.,  
\begin{equation}\label{theorem_G_K_statement}
\lim_{k\to\infty}D_{\mathfrak{f}_n}^{\zeta_{n,k}}(u_{n,l}^\dagger ,u_{n,k}) = 0\;, \qquad \lim_{k\to\infty}\|u_{n,l}^\dagger -u_{n,k}\|=0 \quad \text{almost surely.} 
\end{equation} 
\end{theorem}
\begin{proof}
Consider the event $\Omega_{0}$ defined in \eqref{Prob_one_event}. Then, to establish \ref{theorem_G_K_statement} it is sufficient to show that $D_{\mathfrak{f}_{n}}^{\zeta_{n,k}}(u_{n,l}^{\dagger}, u_{n,k}) \to 0$  along every sample path in $\Omega_{0}$. For this, let we fix an arbitrary sample path $\{i_{k,m} : k \ge 0 \text{ and } 0\leq m < m_{k}\}$ in $\Omega_0$. We examine the iterates along this path and show the convergence using Proposition~\ref{proposition_G_K}. Note that condition (i) of Proposition~\ref{proposition_G_K} holds trivially from definition of $u_{n,k}$, while condition (ii) is a direct consequence of \eqref{mono_G_K_algo_2}. To verify condition (iii), consider the quantity 
\begin{equation}\label{Psi_definition}
\begin{aligned}\Psi_k := &\sum_{m=0}^{m_k-1} S_{k,m} + \sum_{i=1}^{q} \|h_{k,i}\|^2. \\
\end{aligned}
\end{equation}
Applying Lemma~\ref{lemma_prob1_G_K} to this, gives $\sum_{k=0}^{\infty}\Psi_k < \infty$.
This implies that $\Psi_k \to 0$ as $k \to \infty$. Consequently, for each $i\in \{1,\ldots,q\}$, $\|h_{k,i}\| \to 0  \text{ as } k \to \infty$. Now observe that 
\((\mathcal{T}_{i,n}^lu_{n,k}-v_{i,n})\in \operatorname{Ran}(U_{l+1,i})  \text{ and } U_{l+1,i}U_{l+1,i}^*\)
is an orthogonal projection onto $\operatorname{Ran}(U_{l+1,i})$. Therefore, using this with the fact $U_{l+1,i}$ has orthonormal columns, we get
\begin{align*}
\|\mathcal{T}_{i,n}^lu_{n,k}-v_{i,n}\|=\|U_{l+1,i}U_{l+1,i}^*(\mathcal{T}_{i,n}^lu_{n,k}-v_{i,n})\|
&=\|U_{l+1,i}^*(\mathcal{T}_{i,n}^lu_{n,k}-v_{i,n})\|\\
&=\|(A_{k,i})^{1/2}(A_{k,i})^{-1/2}U_{l+1,i}^*(\mathcal{T}_{i,n}^lu_{n,k}-v_{i,n})\|  \\
&=\|(A_{k,i})^{1/2}h_{k,i}\|\\
&\leq\|(A_{k,i})^{1/2}\| \|h_{k,i}\|.
\end{align*}
Since $\gamma_k\leq\gamma_0$ $\forall$ $k\geq0$, we have $(\gamma_kI+B_{l+1,l,i}B_{l+1,l,i}^*)\leq(\gamma_0I+B_{l+1,l,i}B_{l+1,l,i}^*)$, which gives the bound 
\[\|(A_{k,i})^{1/2}\|\leq \|(\gamma_kI+B_{l+1,l,i}B_{l+1,l,i}^*)\|^{1/2}\leq(\gamma_0+\|B_{l+1,l,i}B_{l+1,l,i}^*\|)^{1/2}.\]
Thus, we have 
\[\|\mathcal{T}_{i,n}^lu_{n,k}-v_{i,n}\|\leq (\gamma_0+\|B_{l+1,l,i}B_{l+1,l,i}^*\|)^{1/2}\|h_{k,i}\|.\]
Finally, letting $k\to\infty$ in the above inequality yields \[\| \mathcal{T}_{i,n}^l u_{n,k} - v_{i,n} \| \to 0 \quad \text{as } k \to \infty \] 
for each $i\in \{1,\dots,q\}$.
We now proceed to verify condition (iv) of Proposition~\ref{proposition_G_K}. Since $\Psi_k\to0$ as $k\to\infty$, we can construct a strictly increasing subsequence of non-negative integers $\{k_z\}_{z\geq0}$, such that for every $w\geq z$, 
 \[\Psi_{k_w}\leq\Psi_k \;\text{ for all }\; k_z\leq k\leq k_w-1.\]
Let $\hat{u}_n$ be an arbitrary  solution of \eqref{dis_G_K_mod_eq_exact_da} and $w\geq z$. Note that 
\begin{align*}
\zeta_{n,k_w} - \zeta_{n,k_z} & = \sum_{k=k_z}^{k_w-1} \bigg(\zeta_{n,k,0}-\zeta_{n,k} + \sum_{m=0}^{m_k-1}(\zeta_{n,k,m+1}-\zeta_{n,k,m}) \bigg)\\
&=-\sum_{k=k_z}^{k_w-1}t_k g_k - \sum_{k=k_z}^{k_w-1}\sum_{m=0}^{m_k-1}t_{k,m}g_{k,m}.
\end{align*}
Therefore, 
\begin{equation}\label{inner_product_bound}
\vert\langle \zeta_{n,k_w}-\zeta_{n,k_z}, u_{n,k_w}-\hat{u}_n\rangle\vert\leq \sum_{k=k_z}^{k_w-1}t_k|\langle g_k,u_{n,k_w}-\hat{u}_n \rangle| + \sum_{k=k_z}^{k_w-1}\sum_{m=0}^{m_k-1}t_{k,m}\vert\langle g_{k,m}, u_{n,k_w}-\hat{u}_n\rangle\vert.
\end{equation}
Let $Z:=U_{l+1,i_{k,m}}^*(T_{i_{k,m},n}^lu_{k_w}-v_{i_{k,m},n})$. Then, using the fact that $U_{l+1,i}$ has orthonormal columns with the definitions of $g_{k,m}$ as in Algorithm~\ref{algo:G_K_exact_data_exact} and $T_{i,n}^l$ as in \eqref{Golub_Kahan_approx}, it follows that
\begin{align*}
\vert\langle g_{k,m}, u_{n,k_w}-\hat{u}_n\rangle\vert
&=  \vert\langle V_{l,i_{k,m}}B_{l+1,l,i_{k,m}}^*(A_{i_k,m})^{-1/2}h_{k,m}, u_{n,k_w}-\hat{u}_n\rangle\vert\\
&=\vert\langle h_{k,m}, (A_{i_k,m})^{-1/2}U_{l+1,i_{k,m}}^*U_{l+1,i_{k,m}}B_{l+1,l,i_{k,m}}V_{l,i_{k,m}}^*(u_{n,k_w}-\hat{u}_{n})\rangle\vert \\
&=\vert\langle h_{k,m}, (A_{i_k,m})^{-1/2}U_{l+1,i_{k,m}}^*(T_{i_{k,m},n}^lu_{n,k_w}-v_{i_{k,m},n})\rangle\vert\\
&\leq \|h_{k,m}\|\;\|(A_{i_k,m})^{-1/2}Z\|.
\end{align*}Note that, when $k_z\leq k\leq k_w-1$, the inner step $(k,m)$ occurs before index $k_w$, which implies $\gamma_{k,m}\geq \gamma_{k_w}$. Consequently,
\[ A_{i_{k,m}}=
(\gamma_{k,m}I+B_{l+1,l,i_{k,m}}B_{l+1,l,i_{k,m}}^*)\geq(\gamma_{k_w}I+B_{l+1,l,i_{k,m}}B_{l+1,l,i_{k,m}}^*)=
A_{k_w,i_{k,m}}.\]
Since $A_{i_{k,m}}$ and $A_{k_w,i_{k,m}}$ are strictly positive operators, the above  inequality implies that $\|A_{i_{k,m}}^{-1/2}Z\| \leq\|A_{k_w,i}^{-1/2}Z\|$. Using this with the fact that $i_{k,m}\in \{1,\ldots,q\}$, we obtain
\begin{align*}
\vert\langle g_{k,m}, u_{n,k_w}-\hat{u}_n\rangle\vert
&\leq \|h_{k,m}\|\;\|(A_{k_w,i_{k,m}})^{-1/2}Z\|\\
&\leq\bigg(\sum_{i=1}^{q}\|h_{k,m,i}\|^2\bigg)^{1/2} \bigg(\sum_{i=1}^{q}\|h_{k_w,i}\|^2\bigg)^{1/2}.
\end{align*}
Following similar arguments for $\vert\langle g_k,u_{n,k_w}-\hat{u}_n\rangle\vert$, we also have
\[\vert\langle g_k,u_{n,k_w}-\hat{u}_n\rangle\vert \leq \bigg(\sum_{i=1}^{q}\|h_{k,i}\|^2\bigg)^{1/2} \bigg(\sum_{i=1}^{q}\|h_{k_w,i}\|^2\bigg)^{1/2}.\]
Substituting these estimates in \eqref{inner_product_bound} and using the definitions of $t_k$, $t_{k,m}$, and $\Psi_k$, we obtain
\begin{align*}
\vert\langle \zeta_{n,k_w}-\zeta_{n,k_z}, u_{n,k_w}-\hat{u}_n\rangle\vert
&\leq \mu_1\sum_{k=k_z}^{k_w-1}\bigg(\sum_{i=1}^{q}\|h_{k,i}\|^2\bigg)^{1/2} \bigg(\sum_{i=1}^{q}\|h_{k_w,i}\|^2\bigg)^{1/2}\\
&\qquad\qquad +\mu_1\sum_{k=k_z}^{k_w-1} \sum_{m=0}^{m_k-1}\bigg(\sum_{i=1}^{q}\|h_{k,m,i}\|^2\bigg)^{1/2} \bigg(\sum_{i=1}^{q}\|h_{k_w,i}\|^2\bigg)^{1/2}\\
&\leq\mu_1(\Psi_{k_w})^{1/2}\sum_{k=k_z}^{k_w-1}\bigg[\bigg(\sum_{i=1}^{q}\|h_{k,i}\|^2\bigg)^{1/2} + \sum_{m=0}^{m_k-1}\bigg(\sum_{i=1}^{q}\|h_{k,m,i}\|^2\bigg)^{1/2}\bigg].
\end{align*}
By boundedness of $m_k$ and Cauchy-Schwarz inequality, we have 
\begin{align*}
\bigg(\sum_{i=1}^{q}\|h_{k,i}\|^2\bigg)^{1/2} + \sum_{m=0}^{m_k-1}\bigg(\sum_{i=1}^{q}\|h_{k,m,i}\|^2\bigg)^{1/2}
& \leq \sqrt{m_k+1}\bigg[\sum_{i=1}^{q}\|h_{k,i}\|^2 + \sum_{m=0}^{m_k-1}\sum_{i=1}^{q}\|h_{k,m,i}\|^2\bigg]^{1/2}\\
&\leq \sqrt{M+1}(\Psi_k)^{1/2}.
\end{align*}
Therefore,
\[ \vert\langle \zeta_{n,k_w}-\zeta_{n,k_z}, u_{n,k_w}-\hat{u}_n\rangle\vert\le\mu_1\sqrt{M+1}\sum_{k=k_z}^{k_w-1}\Psi_k^{1/2}\Psi_{k_w}^{1/2}.\]
By construction of the sequence, for every $ k_z\le k\le k_w-1$, we have  $\Psi_{k_w}\le \Psi_k$. Hence $\Psi_k^{1/2}\Psi_{k_w}^{1/2}\le\Psi_k.$
Consequently,
\[ \vert\langle \zeta_{n,k_w}-\zeta_{n,k_z}, u_{n,k_w}-\hat{u}_n\rangle\vert
\le\mu_1\sqrt{M+1}\sum_{k=k_z}^{k_w-1}\Psi_k.\]
Since the series $\sum_{k=0}^{\infty} \Psi_k$ converges, it follows that \[\underset{z\to\infty}{\lim}\underset{w\geq z}\sup{}\left| \langle \zeta_{n,k_w} - \zeta_{n,k_z}, u_{n,k_w} - \hat{u}_n \rangle \right| = 0.\]
Thus, condition (iv) of Proposition~\ref{proposition_G_K} is also satisfied. Hence, by Proposition~\ref{proposition_G_K}, there exists a  solution $u_n^*$ of \eqref{dis_G_K_mod_eq_exact_da} in $\operatorname{dom}(\mathfrak{f}_n)$ such  that $D_{\mathfrak{f}_n}^{\zeta_{n,k}}(u_n^*, u_{n,k} ) \to 0$ as $k\to\infty$.  Consequently, using $\nu$-strong convexity \eqref{bregman_norm_relation} of $\mathfrak{f}_n$, we get $\|u_n^* -u_{n,k}\| \to0$ as $k\to\infty$.
Finally, by the definitions of  $\mathcal{T}_{i_{k,m},n}^{l*}$, $\mathcal{T}_{i,n}^{l*}$ and update rules in Algorithm~\ref{algo:G_K_exact_data_exact} we observe that
\begin{align*}
\zeta_{n,k,m+1} - \zeta_{n,k,m} &= -t_{k,m}\mathcal{T}_{i_{k,m},n}^{l*}U_{l+1,i_{k,m}}(A_{i_k,m})^{-1/2}h_{k,m}, \quad \text{and}\\
 \quad \zeta_{n,k,0} - \zeta_{n,k} &= -t_k\sum_{i \in \Delta_{k}} \mathcal{T}_{i,n}^{l*}U_{l+1,i}(A_{k,i})^{-1/2} h_{k,i}.
\end{align*}
Therefore, we have 
$$\zeta_{n,k,m+1} - \zeta_{n,k,m} \in \operatorname{Ran}(\mathcal{T}_{i_{k,m},n}^{l*}) \subseteq \operatorname{Ran}(\mathcal{T}_{1,n}^{l*})\oplus\ldots\oplus\operatorname{Ran}(\mathcal{T}_{q,n}^{l*})$$ and 
$$\zeta_{n,k,0} - \zeta_{n,k} \in \underset{i \in \Delta_k}{\oplus}\operatorname{Ran}(\mathcal{T}_{i,n}^{l*})\subseteq \operatorname{Ran}(\mathcal{T}_{1,n}^{l*})\oplus\ldots\oplus\operatorname{Ran}(\mathcal{T}_{q,n}^{l*}).$$ 
Consequently, for every $k\geq0,$
\[
\zeta_{n,k+1} - \zeta_{n,k} 
= \left(\zeta_{n,k,0} - \zeta_{n,k}\right) + \sum_{m=0}^{m_k-1}\left(\zeta_{n,k,m+1} - \zeta_{n,k,m}\right) 
\in \operatorname{Ran}(\mathcal{T}_{1,n}^{l*})\oplus\ldots\oplus\operatorname{Ran}(\mathcal{T}_{q,n}^{l*})
.\]
Thus, using the last part of Proposition~\ref{proposition_G_K}, we conclude that $u_n^* = u_{n,l}^\dagger$ along every sample path in $\Omega_0$. Since $\mathbb{P}(\Omega_0) = 1$, the equality holds almost surely. This completes the proof. 
\end{proof}

Our next aim is to establish the following stability result that connects the Algorithms~\ref{algo:G_K_noisy_data} and~\ref{algo:G_K_exact_data_exact}.

\begin{lemma}\label{lemma_stability_G_K}
Consider Algorithms~\ref{algo:G_K_noisy_data} and~\ref{algo:G_K_exact_data_exact} with the same parameters
$\tau>1,\mu_0>0, \mu_1>0 $, the same sequence $\{m_k\}$ and the same
realization of the random indices. Let $\hat{k}:=\lim\inf_{\delta\to0}k_\delta$ along any fixed sample path. Then, for every fixed integer \(k\) with \(0\le k\le \hat k\) and every
\(m=0,\ldots,m_k\), we have
\begin{equation}\label{stability_G_K}
\|u^\delta_{n,k,m}-u_{n,k,m}\|\to0 \;\text{ and } \; \|\zeta^\delta_{n,k,m}-\zeta_{n,k,m}\|\to0\quad\text{as }\delta\to0 .
\end{equation}
When $\hat k=\infty$, we take $0\leq k< \hat{k}$.
\end{lemma}
\begin{proof}
Let us fix an arbitrary sample path $\{i_{k,m}:k\geq0 \text{ and } 0\leq m<m_k\}$. Along this sample path we will use induction on $(k,m)$ to show \eqref{stability_G_K}. We note that $\zeta_{n,0,0}^\delta=\zeta_{n,0,0}$ and $ u_{n,0,0}^\delta = u_{n,0,0}$, so the assertion holds for $(k,m)=(0,0)$. Now to prove the induction step, it is sufficient to verify following two implications:
\begin{enumerate}[(i)]
\item If $u_{n,k,m}^\delta\to u_{n,k,m}$ and $\zeta_{n,k,m}^\delta \to \zeta_{n,k,m}$ as $ \delta\to0 $ for some $0\leq k<\hat{k}$ and $ 0\leq m< m_k$, then $u_{n,k,m+1}^\delta\to u_{n,k,m+1}$ and $\zeta_{n,k,m+1}^\delta\to\zeta_{n,k,m+1}$ as $\delta\to0$.
\item If $u_{n,k}^\delta\to u_{n,k}$ and $\zeta_{n,k}^\delta \to \zeta_{n,k} $ as $ \delta \to0 $ for some $ 1 \leq k < \hat{k} $, then $ u_{n,k,0}^\delta\to u_{n,k,0}$ and $ \zeta_{n,k,0}^\delta \to \zeta_{n,k,0} $ as $ \delta \to0 $.
\end{enumerate}
We start with the proof of~(i). Note that, by induction argument in~(i) and continuity of $\mathcal{T}_{i,n}^l$ for each $i$, we have $h_{k,m}^\delta \to h_{k,m}$ and $g_{k,m}^\delta \to g_{k,m}$ as $\delta \to 0$. Since $k < \hat{k}$, therefore $k < k_\delta$ for all sufficiently small $\delta$. We first show that $t_{k,m}^\delta g_{k,m}^\delta \to t_{k,m}g_{k,m}$. For this we consider two cases. \\
\emph{Case 1.} Suppose $h_{k,m} = 0$. In this case, $g_{k,m} = 0$ and $t_{k,m} = 0$. Since $0 \leq t_{k,m}^\delta \leq \mu_1$, we have
\[
\|t_{k,m}^\delta g_{k,m}^\delta - t_{k,m}g_{k,m}\| = \|t_{k,m}^\delta g_{k,m}^\delta\| \leq \mu_1\|g_{k,m}^\delta\| \to 0.
\]
\emph{Case 2.} Suppose $h_{k,m}\neq0$. Then, $\gamma_{k,m}\|h_{k,m}\|^2>0$. Consequently $\gamma_{k,m}\|h_{k,m}^\delta\|^2\to\gamma_{k,m}\|h_{k,m}\|^2>0$. Since $\tau^2\delta_i^2 \to 0$ as $\delta \to 0$, we have $\gamma_{k,m}\|h_{k,m}^\delta\|^2 > \tau^2\delta_i^2$ for all sufficiently small $\delta$. Thus, 
\[t_{k,m}^\delta=\min\bigg\{\frac{\mu_0\|h_{k,m}^\delta\|^2}{\|g_{k,m}^\delta\|^2},\mu_1\bigg\}.\]
If $g_{k,m}\ne0$, then $g_{k,m}^\delta\to g_{k,m}$ implies $\|g_{k,m}^\delta\|\to\|g_{k,m}\|>0$. 
Therefore, by the continuity of the map $s\mapsto \min\{s,\mu_1\}$, we obtain
\[\min\left\{\frac{\mu_0\|h_{k,m}^\delta\|^2}{\|g_{k,m}^\delta\|^2},\mu_1\right\}\to \min\left\{\frac{\mu_0\|h_{k,m}\|^2}{\|g_{k,m}\|^2},\mu_1\right\}.\]
So, $t_{k,m}^\delta\to t_{k,m}$, which implies $t_{k,m}^\delta g_{k,m}^\delta\to t_{k,m}g_{k,m}$. On the other hand, if $g_{k,m} = 0$, then $g_{k,m}^\delta \to 0$. Thus, $\mu_0\|h_{k,m}^\delta\|^2/\|g_{k,m}^\delta\|^2$ tends to infinity  whenever $g_{k,m}^\delta \neq 0$. Therefore, for all sufficiently small $\delta$, we have $t_{k,m}^\delta = \mu_1$. Since $t_{k,m} = \mu_1$ when $g_{k,m} = 0$, it follows that
\[
t_{k,m}^\delta g_{k,m}^\delta = \mu_1 g_{k,m}^\delta \to 0 = \mu_1 g_{k,m} = t_{k,m}g_{k,m}.
\]
Combining \emph{Case 1} and \emph{Case 2}, we get $t_{k,m}^\delta g_{k,m}^\delta \to t_{k,m} g_{k,m}$ as $\delta\to0$. Using this  together with the assumption that $\zeta_{n,k,m}^\delta \to \zeta_{n,k,m}$, we obtain
\[
\begin{aligned}
\|\zeta_{n,k,m+1}^\delta - \zeta_{n,k,m+1}\| &\leq \|\zeta_{n,k,m}^\delta - \zeta_{n,k,m}\| + \|t_{k,m}^\delta g_{k,m}^\delta - t_{k,m}g_{k,m}\|\to 0 \;\;\text{as } \delta \to 0.
\end{aligned}
\]
Now, using the definitions of iterates, we have
\[ u^\delta_{n,k,m+1}=\nabla\mathfrak \mathfrak{f}_n^*(\zeta^\delta_{n,k,m+1})\;, \quad u_{n,k,m+1}=\nabla\mathfrak \mathfrak{f}_n^*(\zeta_{n,k,m+1}).\]
Hence, from the Lipschitz continuity of $\nabla\mathfrak \mathfrak{f}_n^*$, it follows that
\[\|u_{n,k,m+1}^\delta - u_{n,k,m+1}\| \leq \frac{1}{2\nu}\|\zeta_{n,k,m+1}^\delta - \zeta_{n,k,m+1}\|\to0 \;\;\text{as }\delta\to0.\]
We now prove item (ii) using analogous arguments. By assumption in (ii) and continuity of $\mathcal{T}_{i,n}^l$ for each $i$, we have $h^\delta_{k,i} \to h_{k,i}$ and $g^\delta_{k,i} \to g_{k,i}$, where
\[g_{k,i}^\delta=V_{l,i} B_{l+1,l,i}^{*}
\bigl(A_{k,i}\bigr)^{-1/2}h_{k,i}^\delta \;\text{ and }\; g_{k,i}=V_{l,i} B_{l+1,l,i}^{*}
\bigl(A_{k,i})^{-1/2} h_{k,i}.\]
\emph{Case 1.} Suppose $\Delta_k = \emptyset$. Then $h_{k,i} = g_{k,i} = 0$ for all $i$, which implies $g_k = 0$ and $t_k = 0$. Since $g^\delta_{k,i} \to 0$ and $0 \leq t^\delta_k \leq \mu_1$, we get
\[\|t^\delta_k g^\delta_k - t_k g_k\| = \|t^\delta_k g^\delta_k\| \leq \mu_1 \sum_{i=1}^q \|g^\delta_{k,i}\| \to 0 \quad \text{as } \delta\to0.\]
\emph{Case 2.} Suppose $\Delta_k\neq\emptyset$. Then, if $i\in\Delta_k$, we have
$\|h_{k,i}\|>0$ and hence $\gamma_k\|h_{k,i}\|^2>0$. Since
$h^\delta_{k,i}\to h_{k,i}$ and $\tau^2\delta_i^2\to0$, it follows that $\gamma_k\|h^\delta_{k,i}\|^2>\tau^2\delta_i^2$ for
all sufficiently small $\delta$.
Thus, $\Delta_k\subseteq\Delta^\delta_k$
for all sufficiently small $\delta$. On the other hand, if $i\notin\Delta_k$, then $h_{k,i}=0$. Since
$h^\delta_{k,i}\to h_{k,i}$, we have $h^\delta_{k,i}\to0$, and therefore $ g^\delta_{k,i}\to0$. Note that
\[g^\delta_k=\sum_{i\in\Delta^\delta_k}g^\delta_{k,i}=\sum_{i\in\Delta_k}g^\delta_{k,i}+\sum_{i\in\Delta^\delta_k\setminus\Delta_k}g^\delta_{k,i} \; \text{ and } \;\Phi^\delta_k=\sum_{i\in\Delta^\delta_k}\|h^\delta_{k,i}\|^2=
\sum_{i\in\Delta_k}\|h^\delta_{k,i}\|^2+\sum_{i\in\Delta^\delta_k\setminus\Delta_k}\|h^\delta_{k,i}\|^2.
\]
Therefore, 
\[g_k^\delta\to  \sum_{i\in\Delta_k}g_{k,i}=g_k \quad \text{ and }\quad \Phi_k^\delta\to\sum_{i\in\Delta_k}\|h_{k,i}\|^2=\Phi_k.\]
As $\Delta_k\neq\emptyset$, we have $\Phi_k>0$. If $g_k\neq0$, then $g^\delta_k\to g_k$ implies
$\|g^\delta_k\|\to\|g_k\|>0$. Consequently, 
\[t^\delta_k =\min\left\{
\frac{\mu_0\Phi^\delta_k}{\|g^\delta_k\|^2},\mu_1 \right\}\to\min\left\{
\frac{\mu_0\Phi_k}{\|g_k\|^2},\mu_1\right\} = t_k.
\]
This implies $t^\delta_k g^\delta_k\to t_k g_k$. On the other hand, if $g_k=0$, then $g^\delta_k\to0$. Because $\Phi^\delta_k\to\Phi_k>0$, the quotient
$\mu_0\Phi^\delta_k/\|g^\delta_k\|^2$ tends to infinity whenever
$g^\delta_k\neq0$. Thus, for all sufficiently small $\delta$, we have
$t^\delta_k=\mu_1$. As $t_k = \mu_1$ when $g_k = 0$, it follows that 
\[t^\delta_k g^\delta_k =\mu_1 g^\delta_k\to0=\mu_1 g_k = t_k g_k.
\]
Thus, from \emph{Case 1} and \emph{Case 2},  we have $t^\delta_k g^\delta_k\to t_k g_k$. Therefore, by the assumption $\zeta_{n,k}^\delta \to \zeta_{n,k}$, it follows that
\[
\begin{aligned}
\|\zeta_{n,k,0}^\delta - \zeta_{n,k,0}\| &\leq \|\zeta_{n,k}^\delta - \zeta_{n,k}\| + \|t_{k}^\delta g_{k}^\delta - t_{k}g_{k}\|\to 0 \quad \text{as } \delta \to 0. 
\end{aligned}
\]
Finally, the definitions $u^\delta_{n,k,0}=\nabla\mathfrak \mathfrak{f}_n^*(\zeta^\delta_{n,k,0})$, $u_{n,k,0}=\nabla\mathfrak \mathfrak{f}_n^*(\zeta_{n,k,0})$ and Lipschitz continuity of $\nabla\mathfrak \mathfrak{f}_n^*$ implies 
\[\|u^\delta_{n,k,0}-u_{n,k,0}\|\to0 \quad \text{as } \delta\to0.\]
This completes the proof.
\end{proof}
Building upon the stability result established in Lemma~\ref{lemma_stability_G_K},  monotonicity property in Lemma~\ref{lemma_mon_G_k}, and the convergence property of Algorithm~\ref{algo:G_K_exact_data_exact}, we now establish the convergence result for Algorithm~\ref{algo:G_K_noisy_data}. 

\begin{theorem}\label{theorem_main_G_K}
Consider the Algorithm~\ref{algo:G_K_noisy_data} with a bounded sequence $\{m_k\}$ and suppose the parameters $\tau > 1$ and $\mu_0 > 0$ are chosen so that \eqref{parameter_condition} holds. Then, there exists a random integer $k_\delta$, finite along every sample path, such that the Algorithm~\ref{algo:G_K_noisy_data} terminates after $k_\delta$ iterations. In addition,
\begin{equation}\label{eq:conv_pathwise}
\lim_{\delta \to 0} D_{\mathfrak{f}_n}^{\zeta_{n,k_\delta}^\delta}(u_{n,l}^\dagger, u_{n,k_\delta}^\delta) = 0
 \;, \quad
\lim_{\delta \to 0} \|u_{n,k_\delta}^\delta - u_{n,l}^\dagger\| = 0 \quad\text{almost surely},
\end{equation}
and
\begin{equation}\label{eq:conv_expectation}
\lim_{\delta \to 0} \mathbb{E}\bigg[D_{\mathfrak{f}_n}^{\zeta_{n,k_\delta}^\delta}(u_{n,l}^\dagger, u_{n,k_\delta}^\delta)\bigg] = 0
\;,\quad
\lim_{\delta \to 0} \mathbb{E}\big[\|u_{n,k_\delta}^\delta - u_{n,l}^\dagger\|^2\big] = 0.
\end{equation}
\end{theorem}
\begin{proof}
Let $\{u_{n,k},\zeta_{n,k}\}_{k\ge0}$ be the iterative sequence given by Algorithm~\ref{algo:G_K_exact_data_exact}, with the same parameters $\mu_0, \mu_1$ and the same sequence $\{m_k\}$. Let $\Omega_0$ be the probability one event defined in \eqref{Prob_one_event}. Then, by Theorem~\ref{Theorem_G_K_exact}, along every sample path in $\Omega_0$, we have
\[D^{\zeta_{n,k}}_{\mathfrak{f}_n}(u_{n,l}^\dagger,u_{n,k})\to0,\quad\|u_{n,k}-u_{n,l}^\dagger\|\to0 \quad \text{as } k\to\infty .
\]
Set $\Gamma_{n,k}^\delta := D_{\mathfrak{f}_n}^{\zeta_{n,k}^\delta}(u_{n,l}^\dagger, u_{n,k}^\delta)$. We first show that $\Gamma_{n,k_\delta}^\delta\to0$ as $\delta\to0$ pathwise on $\Omega_0$. For this, let we fix an arbitrary sample path in $\Omega_0$ and define $\hat{k}:= \liminf_{\delta\to0} k_\delta$. We consider the two possible cases.\\
\emph{Case 1.} $\hat{k}<\infty$. In this case $\hat{k}\leq k_\delta$ for sufficiently small $\delta$. Moreover, there exists a sequence of noisy data $\{v_n^{\delta_j}\} = \{(v_{1,n}^{\delta_j}, \ldots, v_{q,n}^{\delta_j})\}$ satisfying
$\|v_{i,n}^{\delta_j} - v_{i,n}\| \leq \delta_i^{j},\; i = 1, \ldots, q,$
where $\delta^j := \sqrt{(\delta_{1}^j)^2+\ldots+(\delta_q^j)^2} \to 0$ as $j \to \infty$, such that $k_{\delta^j} = \hat{k}$ for all sufficiently large $j$. Applying Lemma~\ref{lemma_stability_G_K} along the fixed sample path, we have \[\zeta_{n,\hat{k}}^{\delta^j}\to\zeta_{n,\hat{k}}, \quad u_{n,\hat{k}}^{\delta^j}\to u_{n,\hat{k}} \quad \text{as } j \to \infty.
\] Since the noisy iterations stop as $k_{\delta^j}=\hat{k}$, for each $i=1,\ldots,q$, we have
\[\gamma_{\hat k}\big\|A_{\hat k,i}^{-1/2}U_{l+1,i}^{*}\big(\mathcal{T}_{i,n}^lu_{n,\hat k}^{\delta^j}-v_{i,n}^{\delta^j}\big)\big\|^2\le\tau^2(\delta_i^{j})^2.\]
Letting $j\to\infty$ and using $v_{i,n}^{\delta^j}\to v_{i,n}$, we get $A_{\hat k,i}^{-1/2}U_{l+1,i}^{*}\big(\mathcal{T}_{i,n}^lu_{n,\hat k}-v_{i,n}\big)=0$. 
Since $(A_{\hat{k},i})^{-1/2}$ is strictly positive matrix, we have $U_{l+1,i}^{*}\big(\mathcal{T}_{i,n}^lu_{n,\hat k}-v_{i,n}\big)=0$. Note that $(\mathcal{T}_{i,n}^lu_{n,\hat k}-v_{i,n})\in \operatorname{Ran}(U_{l+1,i})$, therefore it follows that $(\mathcal{T}_{i,n}^lu_{n,\hat k}-v_{i,n})=0$ for each $i=1,\ldots,q$. Consequently, by Algorithm~\ref{algo:G_K_exact_data_exact}, 
\[u_{n,k}=u_{n,\hat{k}} \quad \text{and} \quad \zeta_{n,k}=\zeta_{n,\hat{k}} \quad \forall \; k\geq\hat{k}.\]
Since by Theorem~\ref{Theorem_G_K_exact}, $u_{n,k}\to u_{n,l}^\dagger$ as $k\to\infty$, therefore $u_{n,\hat{k}}=u_{n,l}^\dagger$. Now using the Lemma~\ref{lemma_mon_G_k} and Lemma~\ref{lemma_stability_G_K} together with the definition of lower-semicontinuity, we obtain
\begin{align*}
\limsup_{\delta \to 0}\, \Gamma_{n,k_\delta}^\delta
&\leq \limsup_{\delta \to 0}\, \Gamma_{n,\hat{k}}^\delta \\
&= \limsup_{\delta \to 0} \left(\mathfrak{f}_n(u_{n,l}^\dagger) -\mathfrak{f}_n(u_{n,\hat{k}}^\delta) - \langle \zeta_{n,\hat{k}}^\delta,\, u_{n,l}^\dagger - u_{n,\hat{k}}^\delta \rangle\right) \\
&= \mathfrak{f}_n(u_{n,l}^\dagger) - \mathfrak{f}_n(u_{n,\hat{k}}) - \langle \zeta_{n,\hat{k}},\, u_{n,l}^\dagger - u_{n,\hat{k}} \rangle = 0,
\end{align*}
and hence $\Gamma_{n,k_\delta}^\delta\to 0$ as $\delta\to 0$.\\
\emph{Case 2.} $\hat{k}=\infty$. Let $k\geq0$ be a fixed integer, then $k<k_\delta$ for sufficiently small $\delta$. Applying Lemma~\ref{lemma_mon_G_k}
repeatedly from the $k$-th iterate up to the stopping index $k_\delta$, we obtain $\Gamma^\delta_{n,k_\delta}\le\Gamma^\delta_{n,k}$. Using this with Lemma~\ref{lemma_stability_G_K} and the lower-semicontinuity of $\mathfrak{f}_n$, we get
\[
\limsup_{\delta\to0}\Gamma^\delta_{n,k_\delta}
\leq
\limsup_{\delta\to0}\Gamma^\delta_{n,k}
\leq
D^{\zeta_{n,k}}_{\mathfrak{f}_n}(u_{n,l}^\dagger,u_{n,k}).
\]
Since this holds for every fixed $k$ and by Theorem~\ref{Theorem_G_K_exact}, 
\[D_{\mathfrak{f}_n}^{\zeta_{n,k}}(u_{n,l}^{\dagger}, u_{n,k}) \to 0 \quad \text{as} \quad k \to \infty,\] 
letting $k\to\infty$ we get $\limsup_{\delta\to0}\Gamma^\delta_{n,k_\delta}\leq0$. Hence $\Gamma^\delta_{n,k_\delta}\to0$ as $\delta\to0$.\\
Now, from \emph{Case 1} and \emph{Case 2}  $\Gamma^\delta_{n,k_\delta}\to0$ as $\delta\to0$ along any sample path in $\Omega_0$. Therefore, $\Gamma^\delta_{n,k_\delta}\to0$ as $\delta\to0$ almost surely. Consequently, the $\nu$-strong convexity of $\mathfrak{f}_n$ \eqref{bregman_norm_relation} implies
\[\|u^\delta_{n,k_\delta}-u_{n,l}^\dagger\|\to 0 \quad \text{as}\;\; \delta \to 0 \quad \text{almost surely}.\] 
To prove the convergence in expectation, we note that $\Gamma_{n,k_\delta}^\delta \leq \Gamma_{n,0}^\delta =D_{\mathfrak{f}_n}^{\zeta_{n,0}}(u_{n,l}^{\dagger},u_{n,0})$. Therefore using dominated convergence theorem, $\lim_{\delta\to0}\mathbb E\left[\Gamma_{n,k_\delta}^\delta\right]=0$. Finally, using the $\nu$-strong convexity of $\mathfrak{f}_n$ again, we have $\nu\mathbb E\big[\|u^\delta_{n,k_\delta}-u_{n,l}^\dagger\|^2\big]\le\mathbb E\left[\Gamma_{n,k_\delta}^\delta\right]$. Hence, $\lim_{\delta\to0}\mathbb E\big[\|u^\delta_{n,k_\delta}-u_{n,l}^\dagger\|^2\big]=0$. This completes the proof.
\end{proof}
Our next section is devoted to the analysis of \texttt{RIAT} method. Since the Arnoldi-based scheme has the same iterative structure as its Golub--Kahan counterpart, the convergence analysis can be developed using analogous arguments. To avoid unnecessary repetition, we present the corresponding implementable algorithm and state the principal convergence and regularization results, while omitting proofs that follow directly from the analysis of \texttt{RIGKT}.

\section{The \texttt{RIAT} Framework: Algorithm and Results}\label{section_con_analysis_Arnoldi} 
Building upon the analytical framework introduced in Section~\ref{section_con_analysis_G_K}, this section details the algorithmic implementation of the \texttt{RIAT} method for square systems and states its primary theoretical properties. Specifically, we present the \texttt{RIAT} algorithm alongside its corresponding monotonicity, finite-termination, stability, and regularization results, emphasizing the key modifications introduced by the Arnoldi projection.\\
Throughout this section, all the Golub–Kahan quantities $(\mathcal{T}^l_{i,n},\, V_{l,i},\, U_{l+1,i},\, B_{l+1,l,i})$ are replaced by their respective Arnoldi counterparts $(\tilde{\mathcal{T}}^l_{i,n},\, \tilde{V}_{l,i},\, \tilde{V}_{l+1,i},\, H_{l+1,l,i})$. The noise bound $\Vert{}v_{i,n}^{\delta}-v_{i,n}\Vert{}\leq \delta_i$ for $i=1,\ldots,q$, parameter conditions \eqref{parameter_condition} for $\tau > 1$, $\mu_0, \mu_1 > 0$, and sequence condition \eqref{sequence_condition} for $\{\gamma_j\}_{j\ge 0}$ remain identical to those in Section~\ref{section_con_analysis_G_K}.
We employ the following algorithm for implementing \texttt{RIAT}.

\begin{breakablealgorithm}
\caption{\texttt{RIAT} with \emph{a posteriori} stopping}
\label{algo:Arnoldi_noisy_data}
\begin{algorithmic}[1]
\State Input: $\tau > 1$, $\mu_0 > 0$, $\mu_1 > 0$ and $\zeta_{n,0} \in \mathcal{U}_n$.
\State Compute $u_{n,0} = \arg\min_{u_n \in \mathcal{U}_n}
\bigl\{\mathfrak{f}_n(u_n) - \langle \zeta_{n,0},u_n\rangle\bigr\}$ and set
\Statex \hspace{\algorithmicindent}
$\zeta_{n,0,0}^\delta = \zeta_{n,0}$,\quad
$u_{n,0,0}^\delta = u_{n,0}$.
\For{$k = 0,1,\ldots,$} (\textbf{Outer loop})
\State Pick a fixed positive integer $m_k$.
\For{$m = 0,\ldots,m_k-1$} (\textbf{Inner loop})
\State Draw $i_{k,m} \in \{1,\ldots,q\}$ uniformly at random and compute
\begin{align*}
\tilde h_{k,m}^\delta
&=
\big(\tilde A_{i_{k,m}}\big)^{-1/2}
\tilde V_{l+1,i_{k,m}}^{*}
\big(
\tilde{\mathcal T}_{i_{k,m},n}^{\,l}u_{n,k,m}^\delta
-
v_{i_{k,m},n}^\delta
\big),
\\
\tilde g_{k,m}^\delta
&=
\tilde V_{l,i_{k,m}}
H_{l+1,l,i_{k,m}}^{*}
\big(\tilde A_{i_{k,m}}\big)^{-1}
\tilde V_{l+1,i_{k,m}}^{*}
\big(
\tilde{\mathcal T}_{i_{k,m},n}^{\,l}u_{n,k,m}^\delta
-
v_{i_{k,m},n}^\delta
\big),
\end{align*}
\Statex \hspace{\algorithmicindent}
\hspace{0.5cm} where
$\tilde A_{i_{k,m}}
:=
\gamma_{k,m}I
+
H_{l+1,l,i_{k,m}}H_{l+1,l,i_{k,m}}^{*}.$
\State Set step-size
\[
\tilde t_{k,m}^\delta =
\begin{cases}
\min\!\left\{
\dfrac{\mu_0\,\|\tilde h_{k,m}^\delta\|^2}
{\|\tilde g_{k,m}^\delta\|^2},\;
\mu_1
\right\}
& \text{if } \gamma_{k,m}\|\tilde h_{k,m}^\delta\|^2
> \tau^2\delta_{i_{k,m}}^2, \\[8pt]
0 & \text{otherwise.}
\end{cases}
\]
\State Set
$\zeta_{n,k,m+1}^\delta
=
\zeta_{n,k,m}^\delta
-
\tilde t_{k,m}^\delta\tilde g_{k,m}^\delta$,
and compute
\[
u_{n,k,m+1}^\delta
=
\arg\min_{u_n \in \mathcal{U}_n}
\bigl\{
\mathfrak{f}_n(u_n)
-
\langle \zeta_{n,k,m+1}^\delta,u_n\rangle
\bigr\}.
\]
\EndFor
\State Outer update: Set
$\zeta_{n,k+1}^\delta = \zeta_{n,k,m_k}^\delta$ \text{ and }
$u_{n,k+1}^\delta = u_{n,k,m_k}^\delta$.
\Statex \hspace{\algorithmicindent} For each $i\in\{1,\ldots,q\}$, set 
\[\tilde A_{k,i}: = \gamma_{k}I + H_{l+1,l,i}H_{l+1,l,i}^{*} \text{ and } \tilde h_{k,i}^{\delta}
:= (\tilde A_{k,i})^{-1/2}\tilde V_{l+1,i}^{*}(\tilde{\mathcal T}_{i,n}^{\,l}u_{n,k}^{\delta}-v_{i,n}^{\delta}\bigr).\]
\State Stopping check: \textbf{If}
\begin{equation}\label{DP_Arnoldi}
\gamma_{k+1}\big\|\tilde{h}_{k+1,i}^\delta\big\|^2\leq\tau^2\delta_i^2\quad \forall\, i \in \{1,\ldots,q\},
\end{equation}
\Statex \hspace{\algorithmicindent}
then output $u_{n,k_\delta}^\delta := u_{n,k+1}^\delta$ and stop,
\Statex \hspace{\algorithmicindent}
\textbf{else} compute
\begin{align*}
\tilde\Delta_{k+1}^\delta
&=
\Bigl\{
i \in \{1,\ldots,q\} :
\gamma_{k+1}
\big\|\tilde{h}_{k+1,i}^\delta
\big\|^2
>
\tau^2\delta_i^2
\Bigr\},
\\
\tilde g_{k+1}^\delta
&=
\sum_{i \in \tilde\Delta_{k+1}^\delta}
\tilde V_{l,i}
H_{l+1,l,i}^{*}
\bigl(\tilde A_{k+1,i}\bigr)^{-1/2}\tilde{h}_{k+1,i}^\delta,
\\
\tilde\Phi_{k+1}^\delta
&=
\sum_{i \in \tilde\Delta_{k+1}^\delta}
\big\|\tilde{h}_{k+1,i}^\delta\big\|^2.
\end{align*}
\State Outer step-size: Set
\[
\tilde t_{k+1}^\delta =
\begin{cases}
\min\!\left\{
\dfrac{\mu_0\,\tilde\Phi_{k+1}^\delta}
{\|\tilde g_{k+1}^\delta\|^2},\;
\mu_1
\right\}
& \text{if } \tilde\Delta_{k+1}^\delta \neq \emptyset,\\[8pt]
0
& \text{if } \tilde\Delta_{k+1}^\delta = \emptyset.
\end{cases}
\]
\State Outer update: Set
$\zeta_{n,k+1,0}^\delta
=
\zeta_{n,k+1}^\delta
-
\tilde t_{k+1}^\delta\tilde g_{k+1}^\delta$,
and compute
\[
u_{n,k+1,0}^\delta
=
\arg\min_{u_n \in \mathcal{U}_n}
\bigl\{
\mathfrak{f}_n(u_n)
-
\langle \zeta_{n,k+1,0}^\delta,u_n\rangle
\bigr\}.
\]
\EndFor
\end{algorithmic}
\end{breakablealgorithm}

Now, we state the key results about the algorithm which are useful for establishing the regularization property of \texttt{RIAT} method.
\begin{lemma}\label{lemma_mon_Arnoldi}
Consider the iterative procedure in Algorithm~\ref{algo:Arnoldi_noisy_data} and assume that the parameters $\tau>1$ and $\mu_0>0$ are chosen so that \eqref{parameter_condition} holds. Let $\hat u_n\in \tilde{\mathcal S}_{n,l}$, then, for all $k\leq k_{\delta}$ and $m=0,\ldots,m_k-1$, the following monotonicity estimates hold.
\begin{equation}\label{mono_Arnoldi_algo_1}\tilde\Gamma_{n,k,m+1}^{\delta}-\tilde\Gamma_{n,k,m}^{\delta}\leq -\mathfrak{C}_0\tilde t_{k,m}^{\delta}\|\tilde h_{k,m}^{\delta}\|^2 \quad \text{and}\quad \tilde\Gamma_{n,k+1,0}^{\delta}-\tilde\Gamma_{n,k+1}^{\delta}\leq -\mathfrak{C}_0\tilde t_{k+1}^{\delta}\tilde\Phi_{k+1}^{\delta},
\end{equation}

where $\tilde\Gamma_{n,k,m}^{\delta}:=D_{\mathfrak{f}_n}^{\zeta_{n,k,m}^{\delta}}(\hat u_n,u_{n,k,m}^{\delta}),\quad \text{and}\quad\tilde\Gamma_{n,k}^{\delta}:=D_{\mathfrak{f}_n}^{\zeta_{n,k}^{\delta}}(\hat u_n,u_{n,k}^{\delta}).$
\end{lemma}
\begin{proof}
The proof follows analogous arguments as in Lemma~\ref{lemma_mon_G_k} with the  Golub--Kahan identity replaced by the Arnoldi identity $\tilde{\mathcal T}_{i,n}^{l} =\tilde  V_{l+1,i}H_{l+1,l,i}\tilde V_{l,i}^{*}$. Consequently, the term $\tilde V_{l,i}H_{l+1,l,i}^{*}(\gamma I+H_{l+1,l,i}H_{l+1,l,i}^{*})^{-1/2}$ plays the role of $V_{l,i}B_{l+1,l,i}^{*}(\gamma I+B_{l+1,l,i}B_{l+1,l,i}^{*})^{-1/2}$. Using the estimate \eqref{operator_bounds} together with reasoning similar to the proof of Lemma~\ref{lemma_mon_G_k}, we obtain \eqref{mono_Arnoldi_algo_1}.
\end{proof}

\begin{lemma}\label{lemma_finite_termination_Arnoldi}
Let the parameters $\tau >1  \text{ and } \mu_0>0$ be chosen so that \eqref{parameter_condition} is satisfied. Then, Algorithm~\ref{algo:Arnoldi_noisy_data} terminates in finitely many steps. 
\end{lemma}
\begin{proof}
The proof is analogous to the proof of  Lemma~\ref{lemma_finite_termination_G_K}.
\end{proof}
The Lemma~\ref{lemma_finite_termination_Arnoldi} establishes that the  Algorithm~\ref{algo:Arnoldi_noisy_data} is well-defined and terminates in finite number of iterations, yielding the stopping index $k_{\delta}$. As noted in Section~\ref{section_con_analysis_G_K}, $k_\delta$ remains a random variable owing to its dependence on the realization of the sample paths. We now establish the regularization properties of the iterates generated by Algorithm~\ref{algo:Arnoldi_noisy_data}.
 
\subsection{Regularization property of \texttt{RIAT}}
To establish the regularization property, we adopt the three-stage approach from Section~\ref{section_con_analysis_G_K}: establishing exact-data convergence, proving pathwise stability under data perturbations, and synthesizing the main regularization result.
We begin by defining the exact-data counterpart of Algorithm~\ref{algo:Arnoldi_noisy_data}.
\begin{breakablealgorithm}
\caption{\texttt{RIAT} with exact-data}
\label{algo:Arnoldi_exact_data_exact}
\begin{algorithmic}[1]
\State Input: $\mu_0 > 0$, $\mu_1 > 0$ and $\zeta_{n,0} \in \mathcal{U}_n$.
\State Compute $u_{n,0} = \arg\min_{u_n \in \mathcal{U}_n}
\bigl\{\mathfrak{f}_n(u_n) - \langle \zeta_{n,0}, u_n\rangle\bigr\}$ and set
\Statex \hspace{\algorithmicindent}
$\zeta_{n,0,0} = \zeta_{n,0}$,\quad
$u_{n,0,0} = u_{n,0}$.
\For{$k = 0, 1, \ldots,$} (\textbf{Outer loop})
\State Pick a fixed positive integer $m_k$.
\For{$m = 0, 1, \ldots, m_k - 1$} (\textbf{Inner loop})
\State Draw $i_{k,m} \in \{1,\ldots,q\}$ uniformly at random and compute
\begin{align*}
\tilde h_{k,m}
&=
\bigl(\tilde A_{i_{k,m}}\bigr)^{-1/2}
\tilde V_{l+1,i_{k,m}}^{*}
\bigl(
\tilde{\mathcal{T}}_{i_{k,m},n}^{\,l}u_{n,k,m}
- v_{i_{k,m},n}
\bigr),
\\
\tilde g_{k,m}
&=
\tilde V_{l,i_{k,m}}
H_{l+1,l,i_{k,m}}^{*}
\bigl(\tilde A_{i_{k,m}}\bigr)^{-1/2}
\tilde h_{k,m}.
\end{align*}
\State Set step-size
\[
\tilde t_{k,m} =
\begin{cases}
\min\!\left\{
\dfrac{\mu_0\,\|\tilde h_{k,m}\|^2}
{\|\tilde g_{k,m}\|^2},\;
\mu_1
\right\}
& \text{if } \tilde h_{k,m} \neq 0, \\[8pt]
0 & \text{otherwise.}
\end{cases}
\]
\State Set
$\zeta_{n,k,m+1}
=
\zeta_{n,k,m}
-
\tilde t_{k,m}\tilde g_{k,m}$,
and compute
\[
u_{n,k,m+1}
=
\arg\min_{u_n \in \mathcal{U}_n}
\bigl\{
\mathfrak{f}_n(u_n)
-
\langle \zeta_{n,k,m+1},u_n\rangle
\bigr\}.
\]
\EndFor

\State Outer update: Set
$\zeta_{n,k+1} = \zeta_{n,k,m_k}$ \text{ and }
$u_{n,k+1} = u_{n,k,m_k}$.
\Statex \hspace{\algorithmicindent} For each $i\in\{1,\ldots,q\}$, set $\tilde h_{k,i}
:= (\tilde A_{k,i})^{-1/2}\tilde V_{l+1,i}^{*}(\tilde{\mathcal T}_{i,n}^{l}u_{n,k}-v_{i,n}\bigr).$
\State Stopping check: \textbf{If}
\[
\bigl\|\tilde h_{k+1,i}\bigr\| = 0
\quad \forall\, i \in \{1,\ldots,q\},
\]
\Statex \hspace{\algorithmicindent} then output $u_{n,k+1}$ and stop,
\Statex \hspace{\algorithmicindent} \textbf{else} compute
\begin{align*}
\tilde\Delta_{k+1}
&=
\Bigl\{
i \in \{1,\ldots,q\} :
\bigl\|\tilde h_{k+1,i}\bigr\|^2 > 0
\Bigr\},
\\
\tilde g_{k+1}
&=
\sum_{i \in \tilde\Delta_{k+1}}
\tilde V_{l,i}
H_{l+1,l,i}^{*}
\bigl(\tilde A_{k+1,i}\bigr)^{-1/2}
\tilde h_{k+1,i},
\\
\tilde\Phi_{k+1}
&=
\sum_{i \in \tilde\Delta_{k+1}}
\bigl\|\tilde h_{k+1,i}\bigr\|^2.
\end{align*}

\State Outer step-size: Set
\vspace{-0.4mm}
\[
\tilde t_{k+1} =
\begin{cases}
\min\!\left\{
\dfrac{\mu_0\,\tilde\Phi_{k+1}}
{\|\tilde g_{k+1}\|^2},\;
\mu_1
\right\}
& \text{if } \tilde\Delta_{k+1} \neq \emptyset, \\[8pt]
0
& \text{if } \tilde\Delta_{k+1} = \emptyset.
\end{cases}
\]
\State Outer update: Set
$\zeta_{n,k+1,0}
=
\zeta_{n,k+1}
-
\tilde t_{k+1}\tilde g_{k+1}$,
and compute
\[
u_{n,k+1,0}
=
\arg\min_{u_n \in \mathcal{U}_n}
\bigl\{
\mathfrak{f}_n(u_n)
-
\langle \zeta_{n,k+1,0},u_n\rangle
\bigr\}.
\]
\EndFor
\end{algorithmic}
\end{breakablealgorithm}

Before establishing the  regularization property for \texttt{RIAT}, we need to consider a probability space on which the iterates are defined. We use the same product
probability space as in Section~\ref{section_con_analysis_G_K}. 
In particular the iterates $\zeta_{n,k,m+1}$ and $u_{n,k,m+1}$ generated by Algorithm~\ref{algo:Arnoldi_exact_data_exact} are defined on the finite product probability space $(\Pi_{k,m},\mathcal A_{k,m},\mathbb P_{k,m})$, and the corresponding
natural filtration is considered in the same manner as in
Section~\ref{section_con_analysis_G_K}. Consequently, all random variables arising in Algorithm~\ref{algo:Arnoldi_exact_data_exact} can be realized on the common probability space $(\Omega,\mathcal{F},\mathbb{P})$ and notion of
almost sure convergence is understood with
respect to this probability space.

\begin{lemma}\label{lemma_prob1_Arnoldi}
Consider Algorithm~\ref{algo:Arnoldi_exact_data_exact} and suppose the parameter $\mu_0>0$ is chosen so that \eqref{parameter_condition_2} holds.
Let $u_n^{*}\in\tilde{\mathcal S}_{n,l}$. Then, for all $k\geq0$ and $m=0,\ldots,m_k-1$,
\begin{equation}\label{mono_Arnoldi_algo_2}
\tilde\Gamma_{n,k+1}\leq\tilde\Gamma_{n,k,m+1}\leq \tilde\Gamma_{n,k,m}\leq\tilde\Gamma_{n,k},
\end{equation}
where $\tilde\Gamma_{n,k,m}:=D_{\mathfrak{f}_n}^{\zeta_{n,k,m}}(u_n^{*},u_{n,k,m}),\quad\tilde\Gamma_{n,k}:=D_{\mathfrak{f}_n}^{\zeta_{n,k}}(u_n^{*},u_{n,k}).$ Furthermore, the event
\begin{equation*}\label{Prob_one_event_Arnoldi}
\begin{aligned}
\tilde\Omega_0 := \Bigg\{ \sum_{k=0}^{\infty} \Bigg( & \sum_{m=0}^{m_k-1} \sum_{i=1}^{q}\|\tilde h_{k,m,i}\|^2  + \sum_{i=1}^{q} \|\tilde h_{k+1,i}\|^2 \Bigg) < \infty \Bigg\}
\end{aligned}
\end{equation*}
occurs with probability one, i.e. $\mathbb{P}(\tilde\Omega_0) = 1$, where 
\begin{equation*}\label{h_k_m_i_Arnoldi}
\tilde h_{k,m,i}:=(\tilde{A}_{i_{k,m}})^{-1/2}\tilde V_{l+1,i}^{*} (\tilde{\mathcal{T}}_{i,n}^l u_{n,k,m} - v_{i,n}).
\end{equation*} 
\end{lemma}
\begin{proof}
The proof is identical to the proof of Lemma~\ref{lemma_prob1_G_K}, after
replacing the Golub--Kahan identity by the Arnoldi identity $\tilde{\mathcal T}_{i,n}^{l}=\tilde V_{l+1,i}H_{l+1,l,i} \tilde V_{l,i}^{*}$.
\end{proof}

\begin{theorem}\label{Theorem_Arnoldi_exact}
Consider  Algorithm~\ref{algo:Arnoldi_exact_data_exact} and assume the parameter $\mu_0>0$ is chosen so that \eqref{parameter_condition_2} holds. Let there be a fixed positive integer $M$ with $m_k\leq M$ for all k. Then, the iterates  generated by Algorithm~\ref{algo:Arnoldi_exact_data_exact} converges almost surely to the $\mathfrak f_n$-minimizing solution $\tilde u_{n,l}^\dagger$ of the Arnoldi projected system \eqref{dis_Arnoldi_exact_data}, i.e. 
\begin{equation}\label{theorem_Arnoldi_statement}
\lim_{k\to\infty}D_{\mathfrak{f}_n}^{\zeta_n,k}(\tilde u_{n,l}^\dagger ,u_{n,k}) = 0 \;,\quad \lim_{k\to\infty}\|\tilde u_{n,l}^\dagger -u_{n,k}\|=0 \quad \text{almost surely.} 
\end{equation}
\end{theorem}

\begin{proof}
The proof is analogous to that of Theorem~\ref{Theorem_G_K_exact}. Specifically, we apply Proposition~\ref{proposition_G_K} again with the Golub--Kahan identity replaced by the Arnoldi identity $\tilde{\mathcal T}_{i,n}^{l} =\tilde  V_{l+1,i}H_{l+1,l,i}\tilde V_{l,i}^{*}$,
while Lemma~\ref{lemma_prob1_Arnoldi}  plays the role of Lemma~\ref{lemma_prob1_G_K}.
\end{proof}

\begin{lemma}\label{lemma_stability_Arnoldi}
Consider Algorithms~\ref{algo:Arnoldi_noisy_data} and~\ref{algo:Arnoldi_exact_data_exact} with the same parameters
$\tau>1,\mu_0>0, \mu_1>0 $, the same sequence $\{m_k\}$ and the same
realization of the random indices. Let $\hat{k}:=\lim\inf_{\delta\to0}k_\delta$ along any fixed sample path. Then, for every fixed integer \(k\) with \(0\le k\le \hat k\) and every
\(m=0,\ldots,m_k\), we have
\begin{equation}\label{stability_Arnoldi}
\|u^\delta_{n,k,m}-u_{n,k,m}\|\to0 \;\text{ and } \; \|\zeta^\delta_{n,k,m}-\zeta_{n,k,m}\|\to0\quad\text{as }\delta\to0 .
\end{equation}
When $\hat k=\infty$, we take $0\leq k< \hat{k}$.
\end{lemma}

\begin{proof}
The proof is analogous to the proof of Lemma~\ref{lemma_stability_G_K}.
\end{proof}

\begin{theorem}\label{theorem_main_Arnoldi}
Consider the Algorithm~\ref{algo:Arnoldi_noisy_data} with a bounded sequence $\{m_k\}$ and suppose the parameters $\tau > 1$ and $\mu_0 > 0$ are chosen so that \eqref{parameter_condition} holds. Then, there exists a random integer $k_\delta$, finite along every sample path, such that Algorithm~\ref{algo:Arnoldi_noisy_data} terminates after $k_\delta$ iterations. In addition,
\begin{equation}\label{eq:conv_pathwise_Arnoldi}
\lim_{\delta \to 0} D_{\mathfrak{f}_n}^{\zeta_{n,k_\delta}^\delta}(\tilde u_{n,l}^\dagger, u_{n,k_\delta}^\delta) = 0
\; , \quad
\lim_{\delta \to 0} \|u_{n,k_\delta}^\delta - \tilde u_{n,l}^\dagger\| = 0 \quad\text{almost surely},
\end{equation}
and
\begin{equation}\label{eq:conv_expectation_Arnoldi}
\lim_{\delta \to 0} \mathbb{E}\bigg[D_{\mathfrak{f}_n}^{\zeta_{n,k_\delta}^\delta}(\tilde u_{n,l}^\dagger, u_{n,k_\delta}^\delta)\bigg] = 0 \;,\quad
\lim_{\delta \to 0} \mathbb{E}\big[\|u_{n,k_\delta}^\delta - \tilde u_{n,l}^\dagger\|^2\big] = 0.
\end{equation}
\end{theorem}

\begin{proof}
The proof follows the same reasoning as in Theorem~\ref{theorem_main_G_K}, with 
Lemma~\ref{lemma_mon_Arnoldi}, Lemma~\ref{lemma_stability_Arnoldi} and Theorem~\ref{Theorem_Arnoldi_exact} playing the roles of Lemma~\ref{lemma_mon_G_k}, Lemma~\ref{lemma_stability_G_K} and Theorem~\ref{Theorem_G_K_exact} respectively.
\end{proof}

\section{Numerical simulations and discussions}\label{section_numerics}
This section presents numerical experiments on real-world imaging applications to evaluate the performance and effectiveness of the proposed methods. In particular, we conduct numerical simulations on two-dimensional Computed Tomography (CT), where the discretized forward operators are inherently rectangular and thus directly suited for the \texttt{RIGKT} method. Additionally, we consider an image deblurring problem characterized by a square and symmetric forward operator, which enables a direct performance comparison between \texttt{RIGKT} and \texttt{RIAT} under identical experimental conditions. We will strictly follow the notation of our paper and evaluate the reconstruction quality using Relative Error (RE),  \emph{Peak Signal-to-Noise Ratio (PSNR)}, and \emph{Structural Similarity Index (SSIM)} \cite{wang2004image}. For a reconstructed image $u_{n,k_\delta}^{\delta}$ and true image $u_n^\dagger$,
both normalized to the intensity range $[0,1]$,
\[\text{RE} := \frac{\| u_{n,k_{\delta}}^\delta  - u_n^{\dagger} \|}{\| u_n^{\dagger} \|}\; \text{ and } \quad\text{PSNR} := 20 \log_{10} \left( \frac{1}{\| u_n^{\dagger} - u_{n,k_{\delta}}^\delta \|} \right).\]
In both experiments, we consider the following strongly convex total-variation functional
\begin{equation}\label{strongly_convex_functional}
\mathfrak f_n(u_n) = \tfrac12\|u_n\|^2 + \lambda\,\mathrm{TV}(u_n),
\end{equation}
where $\lambda$ is a parameter that controls the strength of total-variation smoothing. For this functional, the primal step of Algorithms~\ref{algo:G_K_noisy_data} and~\ref{algo:Arnoldi_noisy_data} satisfies
\[
\arg\min_{u_n}\left\{f_n(u_n)-\langle\zeta_n,u_n\rangle
\right\}=\arg\min_{u_n}\left\{\frac12\|u_n-\zeta_n\|^2+\lambda\operatorname{TV}(u_n)\right\}=\operatorname{prox}_{\lambda\operatorname{TV}}(\zeta_n). 
\]     
Therefore, we numerically approximate this step by running Chambolle's dual projection method~\cite{chambolle2004algorithm} for a chosen number of iterations. 

\subsection{X-ray CT}\label{subsection_X-ray_CT}
It is a vital medical imaging technique utilized in the diagnosis of various internal conditions, such as tumors, bone fractures, internal trauma, and hemorrhages.  In a typical CT system, an X-ray source rotates around the subject along a circular or arc-shaped trajectory and emits radiation from a finite set of projection angles. As the X-ray beams traverse the body, they are attenuated to different extents depending on the density and composition of the tissues encountered. The transmitted radiation is subsequently measured by an array of detectors, producing projection data from which cross-sectional images of the internal anatomy are reconstructed. See \cite{natterer2001mathematics} for further details.

\subsubsection{Experimental setup.} We use the Shepp--Logan phantom ($u_n^\dagger$) from \texttt{skimage.data} as our test image discretized on a $128 \times 128$ Cartesian grid ($n_2 = 128^2 = 16384$) and scaled to $[0,1]$. The forward mapping is modeled via a sparse discrete parallel-beam Radon transform featuring $\theta=60$ projection angles uniformly distributed over the interval $[0,\pi)$. For each angular view, the image is rotated using bilinear interpolation on a zero-padded grid and integrated along one coordinate axis. The detector array comprises $P = \lceil\sqrt{2}\times128\rceil+1 = 183$ bins yielding a full system matrix $\mathcal{T}_n$ of dimension $m \times n_2 = 10980 \times 16384$, where $m = P\times \theta$. To facilitate block-iterative reconstruction using Algorithm~\ref{algo:G_K_noisy_data}, the $60$ projection views are partitioned into $q = 30$ consecutive blocks, with each block containing two views. The $i$-th block operator $\mathcal{T}_{i,n}$ is defined as
\[
\mathcal{T}_{i,n}=\begin{bmatrix}\mathcal{T}_{\theta_{2i-1},n}\\[1mm]\mathcal{T}_{\theta_{2i},n}\end{bmatrix}\in\mathbb{R}^{366\times16384},\quad i=1,\ldots,30.
\]
The corresponding exact block data are $v_{i,n}=\mathcal{T}_{i,n}u_n^\dagger$. This construction produces the family of rectangular linear systems to which the Golub--Kahan formulation of Section~\ref{section_con_analysis_G_K} can be  applied. For a prescribed relative noise level,
\[\delta_{rel} = \frac{\|v_{i,n}^{\delta_i}-v_{i,n}\|}{\|v_{i,n}\|},\]
the noisy data within each block are generated independently as
\[v_{i,n}^{\delta_i} = v_{i,n} + e_i, \quad e_i \sim \mathcal{N}(0, \rho_i^2 I_{366}), \quad \rho_i = \delta_{rel} \frac{\|v_{i,n}\|}{\sqrt{366}},
\]
where $I_{366}$ is the identity matrix and the noise bound supplied to Algorithm~\ref{algo:G_K_noisy_data} is the realized value
$\delta_i=\|e_i\|$, so that $\|v_{i,n}^{\delta_i}-v_{i,n}\|\le\delta_i$ by
construction.

\begin{figure}[H]
\centering
\includegraphics[
width=10cm,height=6cm,keepaspectratio
]{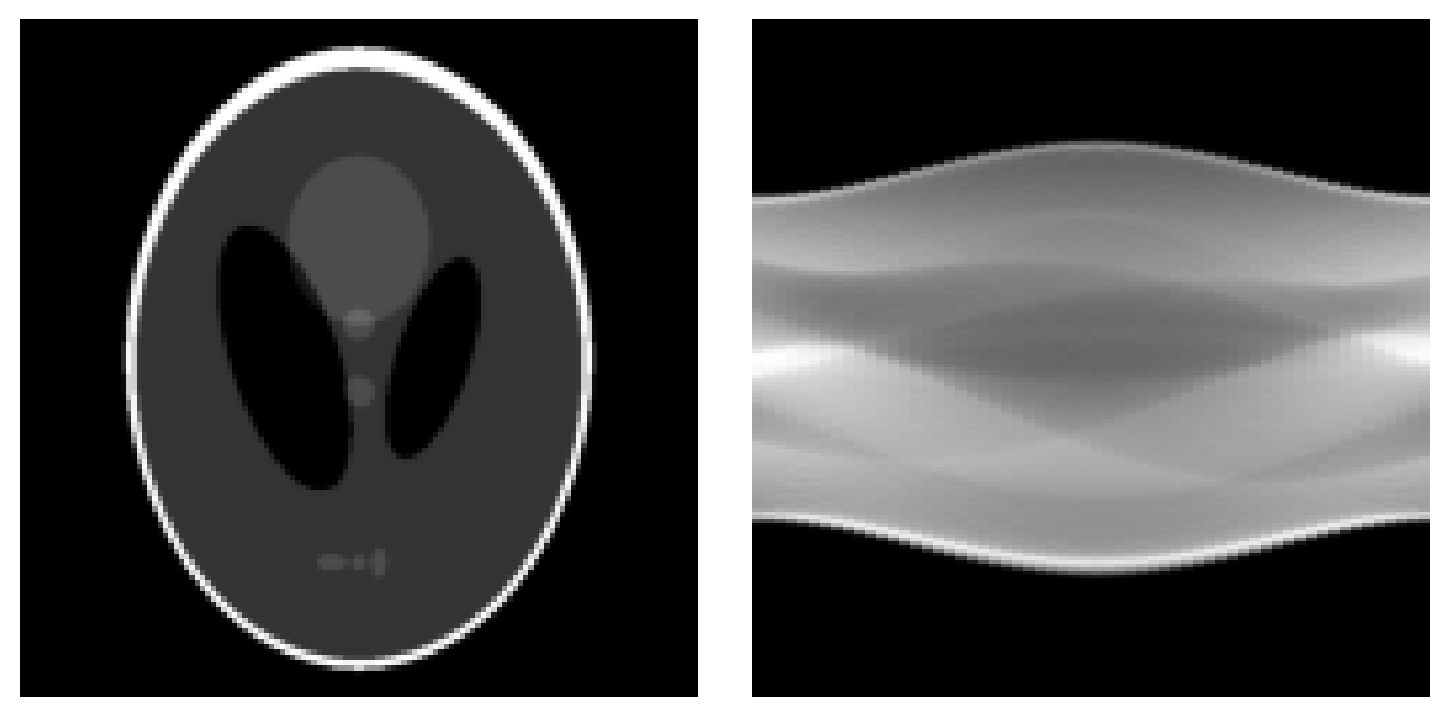}
\caption{True phantom (left) and noisy  sinogram (right) with $\delta_{rel}=0.05$.}
\label{Figure_True_phantom_and_noisy_sinogram}
\end{figure}

For each block we perform $l=80$ steps of Golub--Kahan bidiagonalization
initialized with $v_{i,n}^{\delta_i}$, giving $\mathcal T_{i,n}^l =
U_{l+1,i}B_{l+1,l,i}V_{l,i}^*$ with $B_{l+1,l,i}\in\mathbb R^{81\times80}$.
Since $l\ll\min\{366,16384\}$, the eigen decomposition of
$B_{l+1,l,i}B_{l+1,l,i}^*$ and hence the matrix inverses required in
Algorithm~\ref{algo:G_K_noisy_data} are computed inexpensively at every inner step. We use the strongly convex functional~\eqref{strongly_convex_functional} with $\lambda=0.2$. The primal minimization step in Algorithm~\ref{algo:G_K_noisy_data} is reduced to the proximal TV
problem $\mathrm{prox}_{\lambda\,\mathrm{TV}}(\zeta_n)$, which we solve
approximately using $18$ iterations of the Chambolle's dual projection method \cite{chambolle2004algorithm}. We consider the sequence 
\[\gamma_k = \max\{\gamma_0(0.98)^k, 10^{-4}\}\quad \text{with} \quad
\gamma_0 = \max_{1\le i\le q}\lambda_{\max}(B_{l+1,l,i}B_{l+1,l,i}^*).\]
The value of $\gamma_k$ is updated once per outer iteration and remains fixed during all subsequent inner updates. For the inner iterations we perform $m_k = 12$ randomized updates at each outer index. Finally, the remaining algorithm parameters are set to $\tau=1.15$, $\mu_0=0.1$, and $\mu_1=1.5$, satisfying \eqref{parameter_condition}.

\subsubsection{Reconstruction results.}
The true phantom along with the noisy sinogram is presented in Figure~\ref{Figure_True_phantom_and_noisy_sinogram} and the performance of the proposed \texttt{RIGKT} method across varying noise levels is illustrated in Figure~\ref{Figure_CT_reconstructions}, while Figure~\ref{Figure_CT_relative_error} depicts the relative error versus the iteration count. Additionally, the quantitative metrics for the reconstructed images obtained via Algorithm~\ref{algo:G_K_noisy_data} are summarized in Table~\ref{table:CT-results}.
As expected from the discrepancy principle \eqref{DP_G_K}, the method terminates across all four noise levels. The quantitative and qualitative results indicate that the core structural features and sharp edges of the phantom are well recovered under low-to-moderate noise levels. 
\begin{figure}[H]
\centering
\includegraphics[
width=17cm,height=6cm,keepaspectratio
]{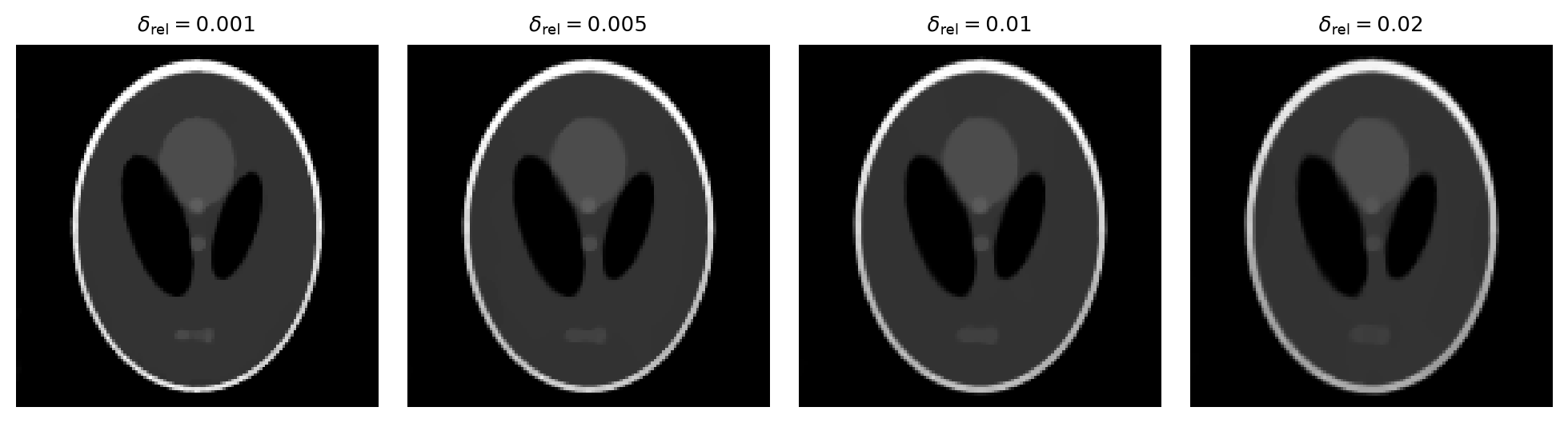}
\caption{Reconstruction results using \texttt{RIGKT} for different noise levels.}
\label{Figure_CT_reconstructions}
\end{figure}

\begin{figure}[H]
\centering
\includegraphics[
width=10cm,height=6cm,keepaspectratio
]{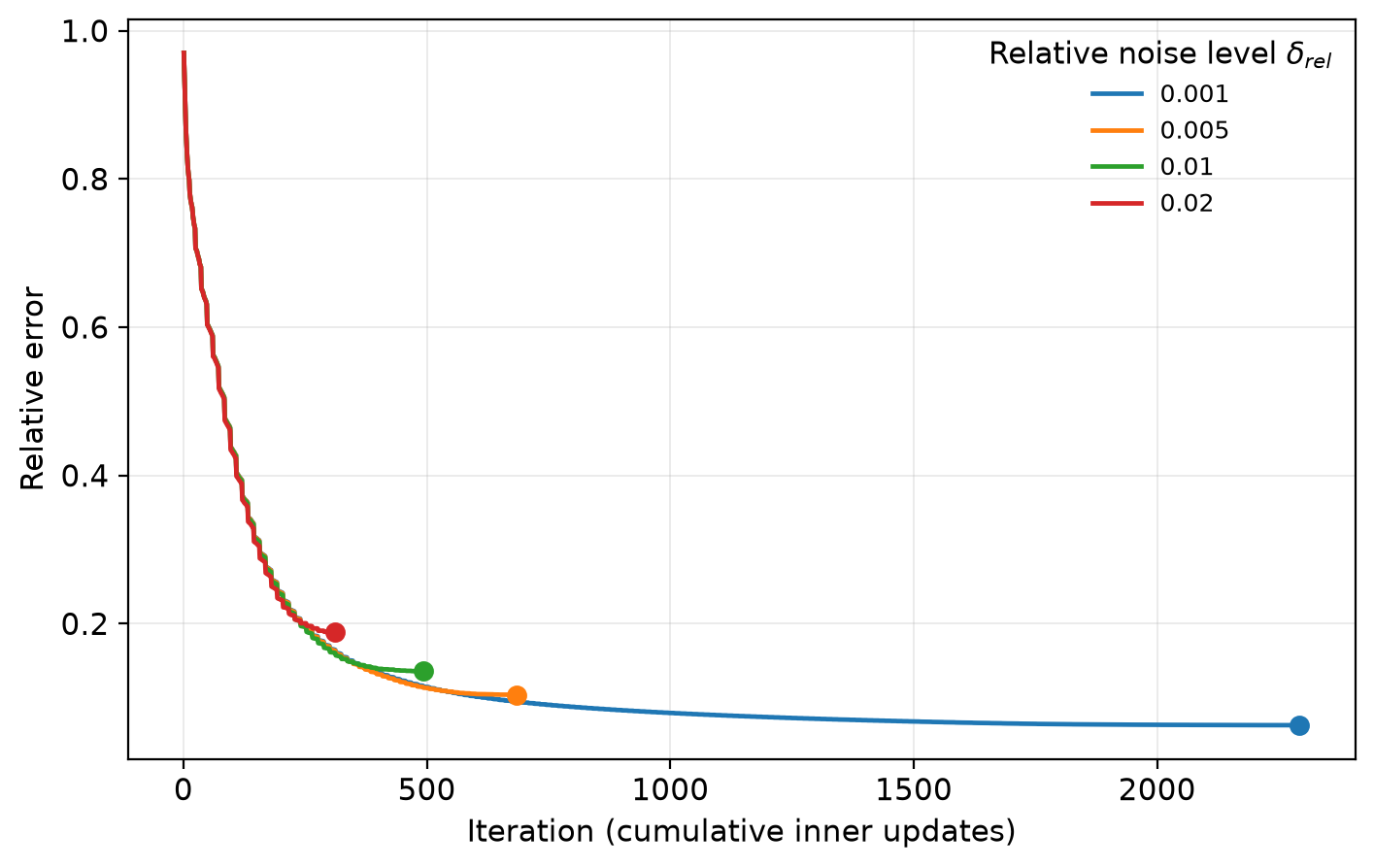}
\caption{Relative error versus iteration for different noise levels.}
\label{Figure_CT_relative_error}
\end{figure}

\begin{table}[H]
\centering
\caption{Reconstruction results for \texttt{RIGKT}.}
\label{table:CT-results}
\begin{tabular}{c@{\quad}c@{\quad}c@{\quad}c@{\quad}c@{\quad}c}
\hline
$\delta_{\rm rel}$ & RE & PSNR (dB) & SSIM & Inner updates & Outer index\\
\hline
$0.001$ & $0.062$ & $36.85$ & $0.995$ & $2292$ & $190$\\
$0.005$ & $0.104$ & $32.35$ & $0.989$ & $684$ & $56$\\
$0.01$ & $0.136$ & $30.01$ & $0.981$ & $492$ & $40$\\
$0.02$ & $0.188$ & $27.17$ & $0.961$ & $312$ & $25$\\
\hline
\end{tabular}
\end{table}

\subsection{Image deblurring}
A typical image deblurring problem can be modeled by a discrete linear ill-posed problem
\[
\mathcal{T}_n u_n^\dagger = v_n,
\]
where $v_n \in \mathbb{R}^{n_2}$ is the blurred image and
$\mathcal{T}_n \in \mathbb{R}^{n_2\times n_2}$ is a forward operator that
models the blurring of a true image $u_n^\dagger \in \mathbb{R}^{n_2}$ represented by $n\times n$ pixels ($n_2 = n^2$). Unlike the X-ray CT example in above subsection where the $q$ block operators $\mathcal{T}_{i,n}$ correspond to 
different rectangular sub-collections of projection angles, here we use a single
common operator with $q$ independent noisy realizations of the same
measurement. This allows us to apply both the Algorithms~\ref{algo:G_K_noisy_data} and \ref{algo:Arnoldi_noisy_data} to the same discretized forward operator and data.
\subsubsection{Experimental setup.} We use the $256\times256$ grayscale \texttt{cameraman} test image from \texttt{skimage.data} ($n=256$, $n_2=65536$), rescaled to the intensity interval $[0,1]$. The forward operator $\mathcal{T}_n$ is a separable
Gaussian blur with reflective boundary conditions constructed as
the Kronecker product $\mathcal{T}_n = D_1 \otimes D_1$ of a  one-dimensional Gaussian blur matrix $D_1 \in \mathbb{R}^{256\times256}$ with standard deviation $\rho_{\mathrm{blur}}=1$.
We take $q=16$ identical block operators $\mathcal T_{i,n}:=\mathcal T_n$, $i=1,\dots,q$, and generate $q$ statistically independent noisy measurements of the same blurred image as
\[
v_{i,n}^{\delta_i} = \mathcal{T}_n u_n^{\dagger} + e_i, \quad
e_i \sim \mathcal{N}(0,\rho^2 I_{n_2}), \quad \rho = \delta_{\mathrm{rel}}\,\frac{\|T_n u_n^{\dagger}\|}{\sqrt{n_2}},
\]
where $\delta_{\mathrm{rel}}$ is the prescribed relative noise level. As in Section~\ref{subsection_X-ray_CT}, the noise bound supplied to the algorithm is 
$\delta_i = \|e_i\|$, so that $\|v_{i,n}^{\delta_i}-v_{i,n}\| \le \delta_i$.
Figure~\ref{Figure_True_image_cameraman_and noisy_observations} shows the true image together with noisy blurred images corresponding to  different noise realizations.

\begin{figure}[H]
\centering
\includegraphics[
width=4.2cm,height=6cm,keepaspectratio
]{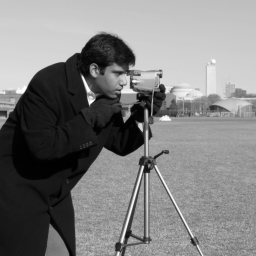}
\includegraphics[
width=17cm,height=6cm,keepaspectratio
]{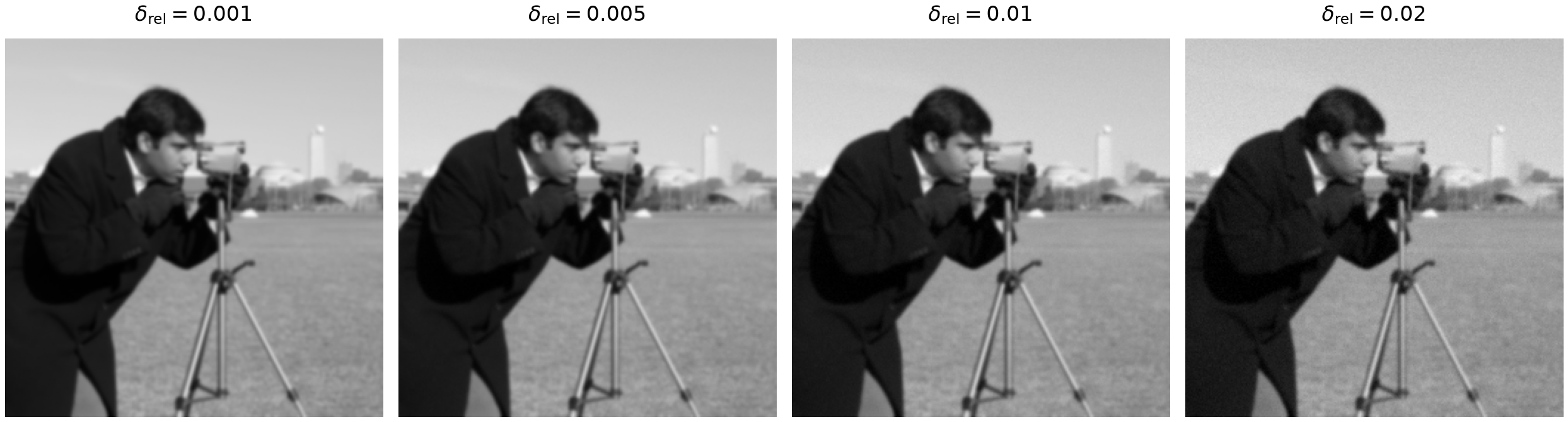}
\caption{True image (top) along with blurred images corresponding to four different noisy realizations (bottom).}
\label{Figure_True_image_cameraman_and noisy_observations}
\end{figure}

For each $i$ we perform $l=200$ steps of both the Golub--Kahan bidiagonalization and the Arnoldi process, initialized with
$v_{i,n}^{\delta_i}$, giving the low-rank approximations
\[\mathcal T_{i,n}^l = U_{l+1,i}B_{l+1,l,i}V_{l,i}^* \quad \text{and }
\tilde{\mathcal T}_{i,n}^l = \tilde V_{l+1,i}H_{l+1,l,i}\tilde
V_{l,i}^*, \quad \text{with } B_{l+1,l,i},H_{l+1,l,i}\in\mathbb R^{201\times200}.\]
Since
$l\ll n_2=65536$, the corresponding $201\times201$ matrix inverses required
in Algorithms~\ref{algo:G_K_noisy_data} and~\ref{algo:Arnoldi_noisy_data} are inexpensive at every inner step. We use the same strongly convex functional \eqref{strongly_convex_functional} with $\lambda=0.005$. The primal minimization step in Algorithms~\ref{algo:G_K_noisy_data} and~\ref{algo:Arnoldi_noisy_data} again reduces to $\mathrm{prox}_{\lambda\,\mathrm{TV}}(\zeta_n)$, which we solve
approximately using $18$ iterations of the Chambolle's dual projection method \cite{chambolle2004algorithm}. The sequence $\{\gamma_k\}$ is chosen as 
\[\gamma_k = \max\{\gamma_0(0.96)^{k},\,10^{-4}\}, \]
where  $\gamma_0 = \max_{1\le i\le q}\lambda_{\max}(B_{l+1,l,i}B_{l+1,l,i}^{*})$ 
for \texttt{RIGKT} and
$\gamma_0 = \max_{1\le i\le q}\lambda_{\max}(H_{l+1,l,i}H_{l+1,l,i}^{*})$ for
\texttt{RIAT}, respectively. We fix $m_k = 8$ and update $\gamma_k$ once per outer iteration while fixing it during all subsequent inner updates similarly as in Section~\ref{subsection_X-ray_CT}. The remaining algorithms parameters are set to $\tau=1.15$, $\mu_0=0.1$, and $\mu_1=1.5$ that satisfy \eqref{parameter_condition}.\\

\subsubsection{Reconstruction results.} We evaluate the reconstruction quality
using the same three measures as in Section~\ref{subsection_X-ray_CT}, (RE), (PSNR), and (SSIM). Figure~\ref{Figure_deblurring_reconstructions_G_k_and_Arnoldi} illustrates the performance of the proposed methods \texttt{RIGKT} and \texttt{RIAT} across varying noise levels, while Figure~\ref{Figure_deblurring_relative_error_RIGK_RIAT} depicts the relative error versus the iteration count. Additionally,  Table~\ref{table:deblurring_results} summarizes the quantitative metrics for the reconstructed images obtained by Algorithms~\ref{algo:G_K_noisy_data} and ~\ref{algo:Arnoldi_noisy_data}. As expected from the discrepancy principles \eqref{DP_G_K} and \eqref{DP_Arnoldi}, both Algorithms~\ref{algo:G_K_noisy_data} and~\ref{algo:Arnoldi_noisy_data} terminate across all four noise levels. The qualitative and quantitative results indicate that the structural features of the true image are well recovered throughout the tested noise range. It should be noted that, because both the schemes operate on the same forward operator $\mathcal{T}_n$, identical noisy data, and the same functional and parameters, they produce nearly identical results, see Figures~\ref{Figure_deblurring_reconstructions_G_k_and_Arnoldi}, \ref{Figure_deblurring_relative_error_RIGK_RIAT} and Table~\ref{table:deblurring_results}. This demonstrates that for a symmetric forward operator, the Golub–Kahan and Arnoldi approximations of $\mathcal{T}_n$ capture comparable information for regularization purposes. Nevertheless, \texttt{RIAT} provides a computational advantage, achieving the same accuracy via a Hessenberg projection that is cheaper to construct than the corresponding bidiagonal one.

\begin{figure}[htbp]
\centering
\includegraphics[
width=17cm,height=6cm,keepaspectratio
]{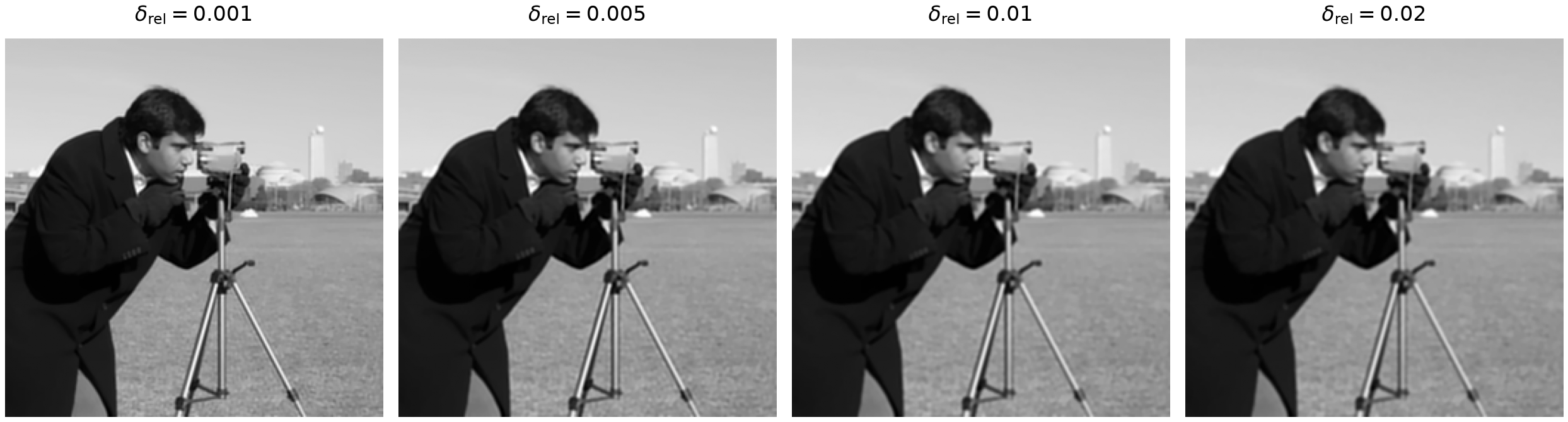}
\includegraphics[
width=17cm,height=6cm,keepaspectratio
]{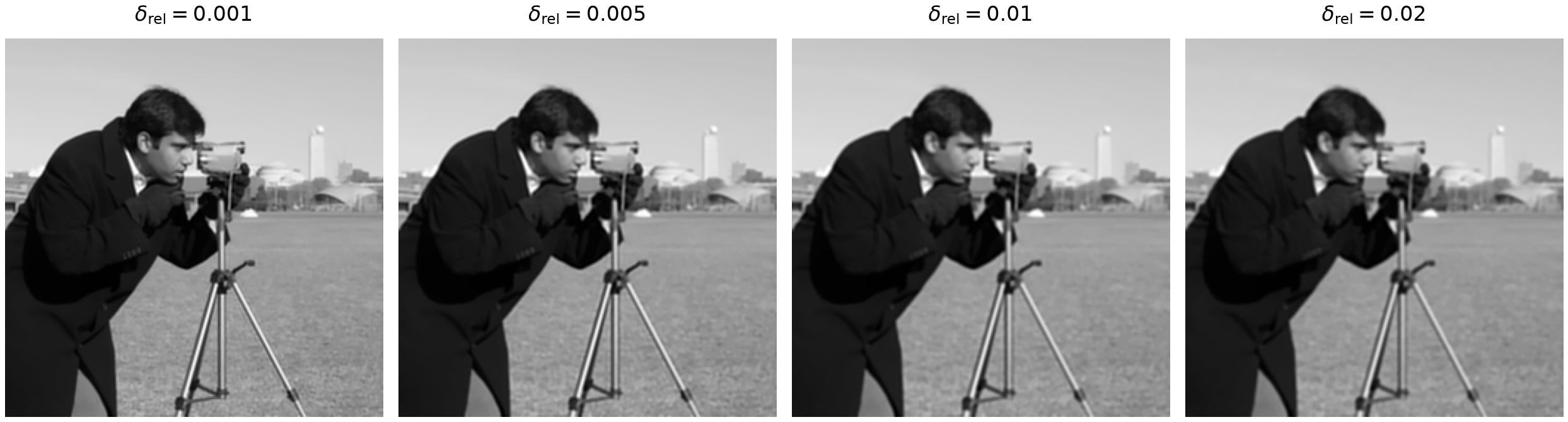}\caption{Reconstructed images using \texttt{RIGKT} (top) and \texttt{RIAT} (bottom).}
\label{Figure_deblurring_reconstructions_G_k_and_Arnoldi}
\end{figure}

\begin{figure}[htbp]
\centering
\includegraphics[
width=17cm,height=6cm,keepaspectratio
]{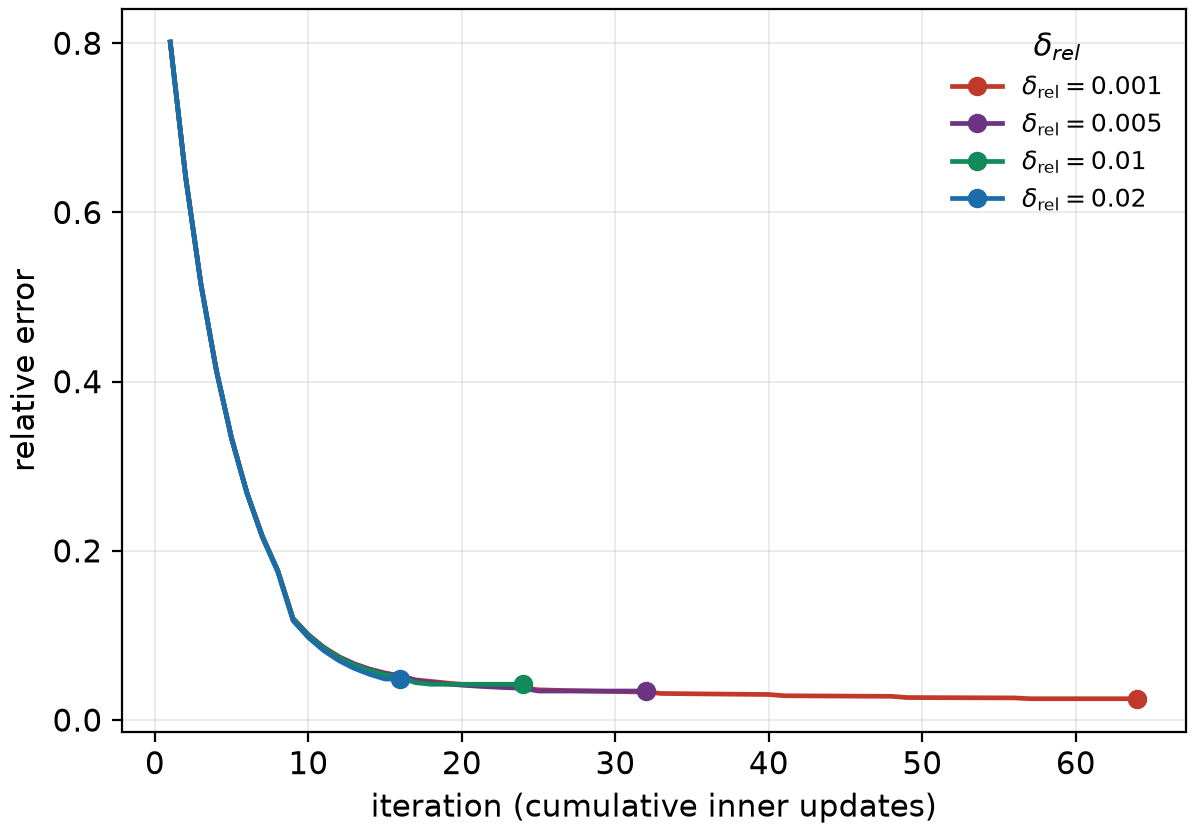}
\includegraphics[
width=17cm,height=6cm,keepaspectratio
]{deblurring_rigkt_relative_error.png}
\caption{Relative error versus iteration for \texttt{RIGKT} (left) and \texttt{RIAT} (right) across varying noise levels.}
\label{Figure_deblurring_relative_error_RIGK_RIAT}
\end{figure}

\begin{table}[htbp]
\centering
\caption{Reconstruction results for \texttt{RIGKT} and \texttt{RIAT}.}
\label{table:deblurring_results}
\begin{tabular}{c@{\quad}c@{\quad}c@{\quad}c@{\quad}c@{\quad}c@{\quad}c}
\hline
Algorithm & $\delta_{\mathrm{rel}}$ & RE & PSNR (dB) & SSIM & Inner updates & Outer index \\
\hline
1 & 0.001 & 0.0252 & 36.73 & 0.962 & 64 & 8 \\
1 & 0.005 & 0.0342 & 34.06 & 0.941 & 32 & 4 \\
1 & 0.01 & 0.0423 & 32.22 & 0.921 & 24 & 3 \\
1 & 0.02 & 0.0485 & 31.04 & 0.908 & 16 & 2 \\
\hline
3 & 0.001 & 0.0252 & 36.71 & 0.962 & 64 & 8 \\
3 & 0.005 & 0.0342 & 34.07 & 0.941 & 32 & 4 \\
3 & 0.01 & 0.0423 & 32.22 & 0.921 & 24 & 3\\
3 & 0.02 & 0.0484 & 31.05 & 0.908 & 16 & 2 \\
\hline
\end{tabular}
\end{table}

\section{Conclusion}\label{conclusion}
In this work we proposed two randomized iterative regularization methods namely, \texttt{RIGKT} and \texttt{RIAT} for large-scale linear inverse problems that integrate randomized iterated Tikhonov regularization with Krylov projections while incorporating strongly convex regularization priors. Theoretically, we established adaptive step-size strategies coupled with discrepancy-based \emph{a posteriori} stopping criterion for both the methods and proved their regularization properties. Numerical simulations on X-ray CT and image deblurring validated the theoretical results, demonstrating that the \emph{a posteriori} stopping criterion consistently recovers the structural features of the exact solution. Notably, the theoretical analysis is particularly challenging and distinct from existing approaches in this framework due to the interplay between the inherent randomness in the iterates and stopping indices, coupled with Krylov subspace projections for dimension reduction.

Several open directions remain for future work. First, establishing explicit convergence rates under the proposed \emph{a posteriori} stopping criterion is of interest, since the randomness of the resulting stopping index makes such an analysis considerably more challenging. Second, the Krylov projection techniques used here could be extended to the broader class of iterative graph Laplacian regularization methods~\cite{Bajpai2026Graph,Bajpai2026Convergence,Bianchi2025DataDependent}, so as to reduce the computational cost of large-scale reconstructions. Third, generalizing the present framework from Hilbert to Banach spaces would broaden its scope to regularization functionals such as total variation and $\ell^1$-type penalties. Fourth, the Gaussian noise model with a known bound assumed for the  discrepancy principle could be relaxed to the distribution-free white-noise setting developed in~\cite{Jahn2023Noise,Harrach2023Regularizing}, building on the Bakushinskii veto~\cite{Harrach2020Beyond} and its circumvention via repeated measurements.

\vspace{1cm}
\textbf{Data Availability.} The data and source code underlying the results of this study can be obtained from the authors  upon reasonable request.

\section*{Acknowledgment}
The authors acknowledge the use of the coding assistant Claude for code debugging and verification.

\bibliographystyle{abbrv}
\bibliography{references}
\end{document}